\documentclass[12pt,reqno]{amsart}
\usepackage{amsmath}
\usepackage{amsfonts}
\usepackage{amssymb}
\usepackage{amsthm}
\usepackage{stmaryrd}
\usepackage{mathrsfs}
\usepackage{mathtools}
\usepackage[makeroom]{cancel}
\usepackage{array}
\usepackage[shortlabels]{enumitem}
\usepackage{hyperref}
\usepackage[margin=1.0in]{geometry}
\usepackage[most]{tcolorbox}
\usepackage{xcolor}
\usepackage{color,soul}
\usepackage{tikz}
\usepackage[latin1]{inputenc}
\usetikzlibrary{
  arrows.meta,shapes,automata,petri,positioning,calc,
  graphs,
  graphs.standard
}
\usepackage{macros}

\begin{document}
\title[Mixed Tensor Products, Capelli Berezinians, and Newton's Formula for $\gl(m|n)$]{Mixed Tensor Products, Capelli Berezinians, and Newton's Formula for $\gl(m|n)$}


\author[S. Erat]{Sidarth Erat}
\address{Sidarth Erat \ La Jolla High School \ La Jolla, CA 92037}
\email{sidarth.erat@gmail.com}
\author[A. S. Kannan]{Arun S. Kannan}
\address{Arun S. Kannan \ Chicago, IL 60606}
\email{arun\_kannan@alum.mit.edu}
\author[S. Kanungo]{Shihan Kanungo}
\address{Shihan Kanungo \ Henry M. Gunn High School \ Palo Alto, CA 94306}
\email{shihankanungo@gmail.com} 
	
\begin{abstract} 
    In this paper, we extend the results of Grantcharov and Robitaille in 2021 on mixed tensor products and Capelli determinants to the superalgebra setting. Specifically, we construct a family of superalgebra homomorphisms $\varphi_R : U(\mathfrak{gl}(m+1|n)) \rightarrow \mathcal{D}'(m|n) \otimes U(\mathfrak{gl}(m|n))$ for a certain space of differential operators $\mathcal{D}'(m|n)$ indexed by a central element $R$ of $\mathcal{D}'(m|n) \otimes U(\mathfrak{gl}(m|n))$. We then use this homomorphism to determine the image of Gelfand generators of the center of $U(\mathfrak{gl}(m+1|n))$. We achieve this by first relating $\varphi_R$ to the corresponding Harish-Chandra homomorphisms and then proving a super-analog of Newton's formula for $\mathfrak{gl}(m)$ relating Capelli generators and Gelfand generators. We also use the homomorphism $\varphi_R$ to obtain representations of $U(\mathfrak{gl}(m+1|n))$ from those of $U(\mathfrak{gl}(m|n))$, and find conditions under which these inflations are simple. Finally, we show that for a distinguished central element $R_1$ in $\mathcal{D}'(m|n)\otimes U(\mathfrak{gl}(m|n))$, the kernel of $\varphi_{R_1}$ is the ideal of $U(\mathfrak{gl}(m+1|n))$ generated by the first Gelfand invariant $G_1$. 
\end{abstract}

\maketitle
\setcounter{tocdepth}{1}
\tableofcontents

\section{Introduction}
\label{sec: intro}

\subsection{The Lie algebra setting} Mixed-tensor type modules, also known as tensor modules or modules of Shen and Larson (see \cite{shen86, larsson92}), are modules over the tensor product algebra $\mathcal{D}'(m) \otimes U(\mathfrak{gl}(m))$, where $\mathcal{D}'(m)$ is a certain algebra of differential operators on the ring $\mathbb{C}[t_0^{\pm 1}, t_1, \dots, t_m]$ and $U(\mathfrak{gl}(m))$ is the universal enveloping algebra of the general linear Lie algebra $\mathfrak{gl}(m)$. Questions related to mixed-tensor type modules have been well studied over the past 35 years (see \cite{shen86, larsson92, rao96, rao04, BilligFutorny16, tz17, he2020class, GS_2021, GR_2021} and references therein). Of interest to us is an algebra homomorphism 
\[U(\mathfrak{sl}(m+1)) \rightarrow \mathcal{D}'(m) \otimes U(\mathfrak{gl}(m))\]
which enables one to inflate representations from $U(\mathfrak{gl}(m))$ to $U(\mathfrak{sl}(m+1))$ (see \cite{GR_2021, GS_2010}). Although this map can be written down algebraically in a straightforward way, there exists a natural geometric interpretation. In particular, it can be thought of as a form of Beilinson-Bernstein localization for parabolic subgroups using a certain maximal parabolic subgroup $P \subset SL(m+1)$ such that $GL(m) \subset P$ and such that $SL(m+1)/P \cong \mathbb{P}^{m}$ (see \cite {GS_2010} and references therein for more details).

In \cite{GR_2021}, the homomorphism $U(\mathfrak{sl}(m+1)) \rightarrow \mathcal{D}'(m) \otimes U(\mathfrak{gl}(m))$ is extended to $U(\mathfrak{gl}(m+1))$. Moreover, the authors explicitly determine the kernel of this extension as well as the image of the Gelfand generators of the center $Z(U(\mathfrak{gl}(m)))$ of $U(\mathfrak{gl}(m))$. Key ingredients in determining the image of the center are Capelli determinants, Capelli generators, and Newton's formula (for all of these, see \cite{Umeda_1998} and \cite{Molev_2007}). Along the way, numerous interesting algebraic identities are also proven; it is expected that these will have implications for the representation theory of tensor modules.

Additionally and importantly to us, the notion of mixed-tensor type modules can be generalized to the setting of Lie superalgebras. Moreover the study of tensor modules in the super setting has received much interest recently (see \cite{BFIK20, LX21, XW23, LX23} and references therein). This setting is the focus of this paper. In particular, we generalize the results of \cite{GR_2021} to the super setting for the general Lie superalgebra $\mathfrak{gl}(m+1|n)$. Most of the arguments therein generalize, although there are some subtleties that arise from the super setting.

\subsection{Results of this paper}
First, we construct a superalgebra of differential operators $\mathcal{D}'(m|n)$ that preserve degree on the superalgebra $\mathbb{C}[t_0^{\pm 1}, t_1, \dots, t_m] \otimes \Lambda_n$, where $\Lambda_n$ is the exterior algebra on $n$ generators with \textit{odd parity.} In Theorem \ref{thm: homomorphism}, we define via explicit formulas a family of superalgebra homomorphisms 
\[\varphi_R: U(\mathfrak{gl}(m+1|n)) \rightarrow \mathcal{D}'(m|n) \otimes U(\mathfrak{gl}(m|n))\] 
indexed by central elements $R$ of $\mathcal{D}'(m|n) \otimes U(\mathfrak{gl}(m|n))$. Moreover, for $m+1 \neq n$ there is a distinguished homomorphism $\varphi_{R_1}$ for a particular central element $R_1$. The map $\varphi_{R_1}$ is an direct generalization to the superalgebra context of the map $\rho$ in \cite{GR_2021} and is ``compatible" with a certain projection homomorphism $\pi^g: U(\mathfrak{gl}(m+1|n)) \rightarrow U(\mathfrak{sl}(m+1|n))$ and the canonical inclusion in the opposite direction (see Proposition \ref{prop: commutative diagram}). 

Using the homomorphism $\varphi_R$, we can consider inflating irreducible representations of $U(\mathfrak{gl}(m|n))$ to $U(\mathfrak{gl}(m+1|n))$ by tensoring with a highest-weight module $\mathcal{F}_a$ over $\mathcal{D}'(m|n)$ generated by $t_0^a$ (such a module can be defined for any $a \in \mathbb{C}$) and pulling back via $\varphi_R$. A partial criteria for when such representations are irreducible is the combined content of Theorems \ref{thm: when the highest weight vector generates the module} and \ref{thm: inflate}. 

We then take a closer look at the restriction of $\varphi_R$ to the center of $U(\mathfrak{gl}(m+1|n))$. In particular, we show that $\varphi_R$ interacts nicely with the Harish-Chandra homomorphism in Theorem \ref{thm : HC homomorphism} and compute the image of the Gelfand generators $G_k^{\mathfrak{gl}(m+1|n)}$ of the center in Theorem \ref{thm : gelfand}. To do the latter, we use the theory of Yangians. In particular, we define what is meant by a Capelli generator in the super setting (originally due to \cite{Nazarov_1991}) and then in Theorem \ref{thm : super newton} we reprove a super Newton's formula (analogous to the formula for $\mathfrak{gl}(m)$ in \cite{Umeda_1998})\footnote{One of the referees pointed out to that generalizations of this formula have been proven several times in the literature, such as in \cite{isaev1999quantum}, \cite{gurevich2004cayley}, \cite{gurevich2006quantum}, and \cite{gurevich2009spectral}, but we reprove it here for completeness.}, which relates the Capelli generators to the Gelfand generators using what we call a \textit{Capelli Berezinian}. Then, by computing the images of the Capelli Berezinian under $\varphi_R$ and applying the super Newton's formula, one can compute the images of Gelfand generators.  An alternative, more computational approach is given in Appendix \ref{long computations}.

Finally, we show that the kernel of $\varphi_{R_1}$ is generated by the first Gelfand invariant $G_1$ in Theorem \ref{thm : kernel}. All our results, except Theorems \ref{thm: when the highest weight vector generates the module} and \ref{thm: inflate}, specialize to those of \cite{GR_2021} when $n = 0$.

\subsection{Organization of this paper} In Section ~\ref{prelims} we review some basics about general linear Lie superalgebras and establish some notation. In Section ~\ref{sec: homomorphism} we construct the homomorphism $\varphi_R$ for $R \in Z(\mathcal{D}'(m|n) \otimes U(\mathfrak{gl}(m|n)))$. In Section~\ref{sec: inflating}, we discuss inflations of representations via the homomorphism $\varphi$.  In Section \ref{sec: harish chandra}, we relate the restriction of $\varphi$ to the center of $U(\mathfrak{gl}(m|n))$ to the corresponding Harish-Chandra homomorphisms. In Section \ref{sec: capelli berezinians}, we reprove the super Newton's formula and determine the images of the Gelfand generators under $\varphi$. Finally, in Section \ref{sec: kernel}, we derive the kernel of the map $\varphi$, a partial generalization of the results in Section 6 of \cite{GR_2021}.

The proof that $\varphi_R$ is a homomorphism is given in Appendix \ref{proof}, which is a long but straightforward computation. Appendix \ref{long computations} gives some explicit formulas for the images under $\varphi$ of a set of homogeneous elements of $U(\mathfrak{gl}(m+1|n))$, using a direct computational approach. All the Gelfand generators can be written as sums of elements from this set, so we obtain an alternative proof of Theorem \ref{thm : gelfand} that does not rely on other sections.

\subsection{Future Directions}
Let us discuss some further directions of inquiry. First of all, unlike the ordinary setting, we do not know of a geometric interpretation of the map $\varphi$ but we believe that the restriction of this map to $U(\mathfrak{sl}(m+1|n))$ should arise from considering a certain ``parabolic subgroup'' $\mathbf{P} \subseteq SL(m+1|n)$ so that the quotient is projective superspace $\mathbb{P}^{m|n}$. In particular, $\mathbf{P}$ should be described by the Harish-Chandra pair  $(\mathbf{P}_0 , \mathfrak{p})$, where the underlying even group $\mathbf{P}_0 \subset \mathbf{P}$ is given by $P$ as above, and the Lie superalgebra $\mathfrak{p}$ of $\mathbf{P}$ satisfies $\mathfrak{p} \cong V \oplus \mathfrak{gl}(m|n)$ where $V$ is the tautological representation of $\mathfrak{gl}(m|n)$ (see \cite{masuoka2012harish, SHIBATA2020179} for theory of Harish-Chandra pairs for supergroups).

Another question is whether there is some larger structure to the family of maps $\varphi_R$. In particular, is there an algebra structure on this set? If so, is there any representation-theoretic interpretation? 

Finally, in \cite{GR_2021} a family of left pseudoinverses are constructed for $\varphi$. It would be interesting to see if something analogous can be done in our setting.

\subsection*{Acknowledgments}
This paper is the result of MIT PRIMES-USA, a program that provides high-school students an opportunity to engage in research-level mathematics and in which the second author mentored the first and third authors. The authors would like to thank the MIT PRIMES-USA program and its coordinators for providing the opportunity for this research experience. We would also like to thank Dimitar Grantcharov for suggesting the project idea and for useful discussions. Additionally, we thank Pavel Etingof, Siddhartha Sahi, Tanya Khovanova, and Thomas R{\"u}d for their thoughts and comments. Finally, we thank the two anonymous referees for their feedback during the review process.

\section{Preliminaries}\label{prelims}
We shall introduce some background for the Lie superalgebra $\mathfrak{gl}(m|n)$ in this section. Throughout the paper we fix nonnegative integers $m,n$. 

\subsection{Basic definitions}

For a Lie superalgebra $\mathfrak{a}$, by $U(\mathfrak{a})$ we denote the universal enveloping superalgebra of $\mathfrak{a}$ and by $Z(\mathfrak{a})$ the center of a superalgebra $\mathfrak{a}$. We will always use $|\cdot|$ to denote the parity of a purely homogeneous element, which is $0$ if even and $1$ if odd (viewed as elements in $\mathbb{Z}/2\mathbb{Z}$). We will work with general linear Lie superalgebra $\mathfrak{gl}(m|n)$, which we define as follows. Let 
\[I \coloneqq \{1,\ldots, m,\bar{1},\ldots,\bar{n}\},\]
\[\hat{I} \coloneqq \{0\}\cup I \]
where we impose the total order 
\[
0<1 < \cdots < m < \bar{1} < \cdots < \bar{n}.
\]
We also define a parity function $| \cdot |$ on $\hat{I}$, where the parity of indices without an overline is $0$ and that of indices with an overline is $1$. Now, let $\mathbb{C}^{m+1|n}$ denote the super vector space with basis $\{e_0, e_1,\dots, e_m, e_{\bar{1}}, \dots, e_{\bar{n}}\}$, where the parity of a basis vector is given by the parity of its index. The general linear Lie superalgebra $\mathfrak{gl}(m+1|n)$ is by definition the space of all linear maps on $\mathbb{C}^{m+1|n}$ and can be thought of as super matrices of the block form  

\begin{equation}\label{eq : blockmatrix }
    \begin{bmatrix}
    \begin{array}{c|c}
    A & B \\
    \hline
    C & D
    \end{array}
    \end{bmatrix},  
\end{equation}
where the matrix $A$ is $(m+1) \times (m+1)$, the matrix $B$ is $(m+1) \times n$, the matrix $C$ is $n \times (m+1)$, and the matrix $D$ is $n \times n$. The underlying even Lie algebra is $\mathfrak{gl}(m+1) \oplus \mathfrak{gl}(n)$ and corresponds to matrices with $B = 0$ and $C = 0$, and the odd part corresponds to matrices such that $A$ and $D$ are zero. The Lie bracket on purely homogeneous elements $x, y \in \mathfrak{gl}(m+1|n)$ is given by 
\[ [x,y] \coloneqq xy - (-1)^{|x||y|}yx\]
where the multiplication is usual matrix multiplication. The \textit{supertrace} $\mathrm{str}$ of a matrix of the form in \eqref{eq : blockmatrix }is defined to be  
\[ \mathrm{str} \left(\begin{bmatrix}
\begin{array}{c|c}
A & B \\
\hline
C & D
\end{array}
\end{bmatrix} \right) = \mathrm{tr}(A) - \mathrm{tr}(D),\]
where $\mathrm{tr}$ denotes the usual trace of a matrix. The supertrace is a Lie superalgebra homomorphism, so by $\mathfrak{sl}(m+1|n)$ we will denote the special linear Lie superalgebra whose elements are the kernel of $\mathrm{str}$. Our standard basis for $\mathfrak{gl}(m+1|n)$ will be given by the usual elementary matrices $\{e_{ij}\}_{i,j \in \hat{I}}$. Finally, \textbf{we will always view $\mathfrak{gl}(m|n)$  as the Lie sub-superalgebra of $\mathfrak{gl}(m+1|n)$ spanned by  $\{e_{ij}\}_{i,j \in I}$}. We emphasize this point as the reader may be confused by the notation in what is to come if otherwise unaware.
\par
A basis for $\mathfrak{sl}(m+1|n)$ is 
\[\{e_{ij}: i,j\in \hat{I}, i\ne j\}\cup \{e_{00}-(-1)^{|i|}e_{ii}: i\in I\}\]
and we shall view $\mathfrak{sl}(m|n)$ as the elements $\mathfrak{sl}(m+1|n)$ that also lie in $\mathfrak{gl}(m|n)$ under the inclusion above. Some other important subalgebra for $\mathfrak{gl}(m+1|n)$ will be $\mathfrak{h}_{m+1|n}$, the Cartan subalgebra of diagonal matrices in $\mathfrak{gl}(m+1|n)$, and $\mathfrak{n}^{-}_{m+1|n}$ and $\mathfrak{n}^+_{m+1|n}$, the subalgebras of strictly lower and upper triangular matrices in $\mathfrak{gl}(m+1|n)$, respectively. We define the analogs of these subalgebras for $\mathfrak{gl}(m|n)$ in the obvious way.
\par
For a square matrix $A$ and a variable or constant $v$, the expression $A+v$ should be understood as the sum of $A$ and the scalar matrix of the same size as $A$ having $v$ on the diagonal. Finally, let $\delta_{kl}$ denote the Kronecker delta, which evaluates to $1$ if $k=l$ and $0$ otherwise.

In Section~\ref{sec: homomorphism}, $a$ and $b$ denote elements of $I$. In all sections, $i$ and $j$ denote elements of $I$ unless specified otherwise.
\subsection{Weights and Root Systems}\label{sec: prelims}

Let $\{\delta_0, \delta_1,\ldots,\delta_m,\delta_{\bar{1}},\ldots, \delta_{\bar{n}}\}$ denote the basis of $\mathfrak{h}_{m+1|n}^*$ dual to the canonical basis $\{e_{00}, e_{11},\ldots, e_{mm},e_{\bar{1}\bar{1}},\ldots, e_{\bar{n}\bar{n}}\}$ of $\mathfrak{h}_{m+1|n}^*$. Whenever convenient, for $1 \leq i \leq n$, we will write 
\[ \varepsilon_i:=\delta_{ \bar{i}}.\]
We will refer to elements of $\mathfrak{h}_{m+1|n}^*$ as \textit{weights}, and will often write any weight $\lambda = \sum_{i \in \hat{I}} \lambda_i \delta_i$ as a tuple $(\lambda_0,\lambda_1, \dots, \lambda_m, \lambda_{\bar{1}}, \dots, \lambda_{\bar{n}})$. The bilinear form $(\cdot, \cdot): \mathfrak{h}_{m+1|n} \times \mathfrak{h}_{m+1|n} \longrightarrow \mathbb{C}$ given by the supertrace $(x, y) \coloneqq \mathrm{str}(xy)$ naturally induces a bilinear form on $\mathfrak{h}_{m+1|n}^*$, which will also be denoted by $(\cdot, \cdot)$. For $i,j\in \hat{I}$, we have
\[
(\delta_i, \delta_j) = (-1)^{|i||j|}\delta_{ij}.
\]
We call $x \in \mathfrak{gl}(m+1|n)$ a \textit{root vector} if there exists a nonzero $\alpha \in \mathfrak{h}_{m+1|n}^*$ if for all $h \in \mathfrak{h}_{m+1|n}$ we have $[h, x] = \alpha(h)x$. The weight $\alpha$ is called a \textit{root}, and set of all roots is called the \textit{root system}. It is well known that a root system for $\mathfrak{gl}(m+1|n)$ is given by $\Phi \coloneqq \{\delta_i - \delta_j\}_{i \neq j \in \hat{I}}$, with a root vector corresponding to $\delta_i - \delta_j$ being $e_{ij}$. Let $\Phi_{\bar{0}}$ and $\Phi_{\bar{1}}$ be the even and odd roots in $\Phi$, respectively, where a root is even if the corresponding root vector is even (resp. odd). It is easily seen that 
\[{\Phi}_{\bar{0}} = \{\delta_i - \delta_j \}_{0 \leq i \neq j \leq m}\cup \{\varepsilon_i - \varepsilon_j \}_{1 \leq i \neq j \leq n},\] 
\[{\Phi}_{\overline{1}} = \{\pm(\delta_i - \varepsilon_j)\}_{0 \leq i \leq m, 1 \leq j \leq n}.\] 
A root $\alpha \in \Phi$ is said to be \textit{isotropic} if $(\alpha, \alpha) = 0$. The set of odd roots coincides with the set of isotropic roots for the general linear Lie superalgebras. Let $\Phi^+ = \{\delta_i - \delta_j\ | \ i < j \in \hat{I}\}$ be a positive system, and restrict this notation to the even and odd roots in the obvious way. Lastly, let $\mathcal{W}_{m+1|n} \cong S_{m+1} \times S_n$ denote the Weyl group of $\mathfrak{gl}(m+1|n)$ with natural action on $\mathfrak{h}_{m+1|n}^{*}$, i.e, $S_{m+1}$ permutes $\{\delta_0, \dots, \delta_m\}$ and $S_n$ permutes $\{\varepsilon_1, \dots, \varepsilon_n\}$.

Furthermore, we can define for any $\alpha \in \Phi_{\overline{0}}$ the corresponding coroot $\alpha^\vee \in \mathfrak{h}_{m+1|n}$ such that 
\[
\langle \lambda, \alpha^\vee \rangle = \frac{2(\lambda, \alpha)}{(\alpha, \alpha)}   \ \ \forall \lambda \in \mathfrak{h}_{m+1|n}^*.
\]
The simple reflection $s_\alpha$ acts on $\mathfrak{h}_{m+1|n}^*$ as expected: $s_\alpha(\lambda) = \lambda - \langle \lambda, \alpha^\vee \rangle \alpha$. Define the Weyl vector $\rho_{m+1|n}$ as follows:
\[\label{normWeyl}
\rho_{m+1|n} \coloneqq \sum_{i=0}^m (m-i+1)\delta_i - \sum_{j=1}^n j\varepsilon_j = (m+1,m,\ldots, 1, -1,-2,\ldots, -n).
\] 
Note that our Weyl vector is a shifted version of the standard definition, which is half the sum of the positive even roots minus half the sum of the positive odd roots, but all this changes is that the Harish-Chandra homomorphism gets shifted. A weight $\lambda \in \mathfrak{h}_{m+1|n}^*$ is said to be \textit{antidominant} if $\langle \lambda + \rho_{m+1|n}, \alpha^\vee\rangle \not\in \mathbb{Z}_{> 0}$ for all $\alpha \in \Phi^+_{\overline{0}}$. An element $\lambda \in \mathfrak{h}_{m+1|n}^*$ is said to be \textit{typical} (relative to $\Phi^+$) if there is no positive isotropic $\alpha \in \Phi_{\overline{1}}$ such that $(\lambda + \rho_{m+1|n}, \alpha) = 0$.
\par

For a weight $\lambda = (\lambda_0,\ldots, \lambda_m, \lambda_{\bar{1}},\ldots, \lambda_{\bar{n}})$ of $\mathfrak{gl}(m+1|n)$, denote by $M(\lambda)$ and $L(\lambda)$ the Verma module of highest weight $\lambda$ and its unique simple quotient, respectively. Here
\[M(\lambda) = U(\mathfrak{gl}(m+1|n))\otimes_{U(\mathfrak{h}_{m+1|n} \oplus \mathfrak{n}^+_{m+1|n})} \mathbb{C}_\lambda,\]
where $\mathbb{C}_\lambda$ is the one-dimensional weight representation of $\mathfrak{h}_{m+1|n}$ with weight $\lambda$ extended trivially to $\mathfrak{n}^+_{m+1|n}$.
\par
Finally (and very importantly to avoid confusion later on), to be consistent with our viewing $\mathfrak{gl}(m|n)$ as the subalgebra of $\mathfrak{gl}(m+1|n)$ spanned by $(e_{ij})_{i,j\in I}$, we restrict all of the above definitions to $\mathfrak{gl}(m|n)$ in the obvious way so that everything is compatible with the inclusion $\mathfrak{gl}(m|n) \xhookrightarrow{} \mathfrak{gl}(m+1|n)$ (with the ``extra" index being $0$). For instance,
\begin{enumerate}
    \item we will simply write a weight $\lambda \in \mathfrak{h}_{m|n}^*$ of $\mathfrak{gl}(m|n)$ as an $(m+n)$-tuple, understanding that $\lambda_0 = 0$ under the inclusion $\mathfrak{h}_{m|n}^* \xhookrightarrow{} \mathfrak{h}_{m+1|n}^*$;
    \item the Weyl group $\mathcal{W}_{m|n}$ is $S_m \times S_n$ and the Weyl vector $\rho_{m|n}$ is
    \[\rho_{m|n} \coloneqq (m,m-1,\ldots, 1, -1,-2,\ldots, -n);\]
    \item a Verma module $M(\lambda)$ and its simple quotient $L(\lambda)$ can either mean such a $\mathfrak{gl}(m+1|n)$-module with highest weight $\lambda$ or such a $\mathfrak{gl}(m|n)$-module with highest weight $\lambda$ (not by restriction). Context will make it clear to which we refer (using a $(m+n)$-tuple will mean interpret it as a module over $\mathfrak{gl}(m|n)$, for instance).
\end{enumerate}
 
\subsection{Superalgebra of Differential Operators}
In this subsection, we define the superalgebra $\mathcal{D}'(m|n)$ of differential operators. 
First, we define the superalgebra of polynomials that $\mathcal{D}'(m|n)$ operates on.

Define $\mathcal{F} \coloneqq \mathbb{C}[t_0^{\pm 1},t_1,\dots, t_m] \otimes \Lambda_n$, where $\Lambda_n$ is the $n$-generator exterior algebra with odd generators $t_{\bar{1}},\ldots, t_{\bar{n}}$. 
The defining relations of $\Lambda_n$ are given by
\[t_it_j = - t_jt_i\]
for $i,j\in \{\bar{1},\ldots,\bar{n}\}$. This anticommutativity, combined with the fact that $\mathbb{C}[t_0^{\pm 1},t_1,\dots, t_m]$ is commutative, means that $\mathcal{F}$ is supercommutative, i.e. that is $x$ and $y$ are homogeneous, then $xy = (-1)^{|x||y|}yx$. A basis for $\mathcal{F}$ is given by vectors of the form $t_0^{k_0}t_1^{k_1}\cdots t_m^{k_m}t_{\bar{1}}^{k_{\bar{1}}}\cdots t_{\bar{n}}^{k_{\bar{n}}}$, where $k_0 \in \mathbb{Z}$, $k_1, \dots, k_m \in \mathbb{Z}_{\geq 0}$, and $k_{\bar{1}}, \dots k_{\bar{n}} \in \{0,1\}$. For convenience, we will drop the tensor symbol in the notation and also freely permute the $t_i$'s in accordance with the supercommutativity rules.

Next, we define the operation $\tfrac{\partial}{\partial t_i}$. We write $\partial_i$ as shorthand for $\tfrac{\partial}{\partial t_i}$. 
Consider an element $t_0^{k_0}t_1^{k_1}\cdots t_m^{k_m}t_{\bar{1}}^{k_{\bar{1}}}\cdots t_{\bar{n}}^{k_{\bar{n}}}\in \mathcal{F}$. To operate by $\partial_i$, we first move $t_i^{k_i}$ to the front using supercommutativity. Then replace $t_i^{k_i}$ by $k_it_i^{k_i-1}$. Extend to the whole of $\mathcal{F}$ by linearity and supercommutativity.

Then, we define the superalgebra of differential operators $\mathcal{D}'(m|n)$ to be the superalgebra generated by: left multiplication by $\frac{t_i}{t_0}$ for $i\in \hat{I}$, where we will abuse notation and simply write $\frac{t_i}{t_0}$ for this operation; and $t_0\partial_i$ for $i\in \hat{I}$, which consists of left applying the $i$'th derivative and then left multiplication by $t_0$. The parity of these is operators is given by the parity of $i \in \hat{I}$.  We define the element 
\[\mathcal{E} \coloneqq \sum_{i \in \hat{I}} \left(\frac{t_i}{t_0}\right)(t_0\partial_i)  =\sum_{i\in \hat{I}} t_i\partial_i \in \mathcal{D}'(m|n).\]
Note that $\mathcal{E}$ generates the center of $\mathcal{D'}(m|n).$

The following lemma collects some identities in $\mathcal{D}'(m|n)$. The proofs are simple so we omit them.
\begin{lemsub}\label{lem: identities for D'(m|n)}
    The superalgebra $\mathcal{D}'(m|n)$ of differential operators on $\mathbb{C}[t_0^{\pm 1},t_1,\dots, t_m] \otimes \Lambda_n$ generated by $\frac{t_i}{t_0}$ for $i\in I$ and $t_0\partial_i$ for $i\in \hat{I}$ satisfies the following properties:
    \begin{enumerate}
        \item We have 
        \begin{align*}
            t_it_j&= (-1)^{|i||j|}t_jt_i,\tag{for $i,j\in \hat{I}$}\\
            \partial_it_j &= (-1)^{|i||j|}t_j\partial_i, \tag{for $i, j\in \hat{I}$, and $ i \ne j$}\\
            \partial_i\partial_j&= (-1)^{|i||j|}\partial_j\partial_i. \tag{for $i,j\in \hat{I}$, and $i\ne j$} 
        \end{align*}
        \item We have
        \[[\partial_a, t_a]=\partial_at_a - (-1)^{|a|}t_a\partial_a= 1\]
        for $a\in \hat{I}$.
        \item The element $\mathcal{E}$ is central in $\mathcal{D}'(m|n)$ and if $p = t_0^{e_0}t_1^{e_1}\cdots t_m^{e_m}t_{\bar{1}}^{e_{\bar{1}}}\cdots t_{\bar{n}}^{e_{\bar{n}}}$, then $\mathcal{E} p=(\deg p)p$, where $\deg p = e_0 + e_1+\cdots + e_m + e_{\bar{1}}+\cdots + e_{\bar{n}}$.
        \item $Z(\mathcal{D}'(m|n)\otimes U(\mathfrak{gl}(m|n))) = \mathbb{C}[\mathcal{E}] \otimes Z(U(\mathfrak{gl}(m|n)))$.
    \end{enumerate}
\end{lemsub}

\section{The homomorphism \texorpdfstring{$\varphi$}{TEXT}}
\label{sec: homomorphism}
In this section, we construct the homomorphism \[\varphi_R: U(\mathfrak{gl}(m+1|n))\to \mathcal{D}'(m|n)\otimes U(\mathfrak{gl}(m|n))\] that extends the homomorphism $\rho$ in~\cite{GR_2021}. 

\begin{thm}\label{thm: homomorphism}
    For any element $R$ in $Z(\mathcal{D}'(m|n)\otimes U(\mathfrak{gl}(m|n)))= \mathbb{C}[\mathcal{E}] \otimes Z(U(\mathfrak{gl}(m|n)))$, the correspondence given by
    \begin{align*}
        e_{ab}&\mapsto t_a\partial_b\otimes 1 + 1\otimes e_{ab}+\delta_{ab}(-1)^{|a||b|}R,  \tag{$a,b\in I$}\\
        e_{a0}&\mapsto t_a\partial_{0}\otimes 1 -\sum_{j\in I} (-1)^{|a||j|}\frac{t_j}{t_0}\otimes e_{aj}, \tag{$a\in I$} \\
        e_{0b}&\mapsto t_0\partial_b\otimes 1, \tag{$b\in I$} \\
        e_{00}&\mapsto t_0\partial_0 \otimes 1 +R
    \end{align*}
    extends by the universal property to a homomorphism $\varphi_R: U(\mathfrak{gl}(m+1|n))\to \mathcal{D}'(m|n)\otimes U(\mathfrak{gl}(m|n))$.
\end{thm}
The proof is given in Appendix \ref{proof}; it is a long but straightforward verification of commutators. We will write $\varphi$ as shorthand for $\varphi_R$ if we are referring to an arbitrary $R$.
Now, let $\varphi^s$ denote the restriction of $\varphi_R$ to $U(\mathfrak{sl}(m+1|n))$. The formulas for $\varphi_R$ give us the following:
\begin{cor}\label{cor: sl}
    The correspondence given by
    \begin{align*}
        e_{ab} &\mapsto t_a\partial_b\otimes 1 + 1\otimes e_{ab}, \tag{$a,b\in I$ and $a\ne b$}\\
        e_{a0}&\mapsto t_a\partial_{0}\otimes 1 -\sum_{j\in I} (-1)^{|a||j|}\frac{t_j}{t_0}\otimes e_{aj},  \tag{$a\in I$}\\
        e_{0b}&\mapsto t_0\partial_b\otimes 1,\tag{$b\in I$}\\
        e_{00}-(-1)^{|a|}e_{aa} &\mapsto (t_0\partial_0-(-1)^{|a|}t_a\partial_a)\otimes 1 - 1\otimes (-1)^{|a|}e_{aa} \tag{$a\in I$}
    \end{align*}
    extends by the universal property to a homomorphism $\varphi^s: U(\mathfrak{sl}(m+1|n))\to \mathcal{D}'(m|n)\otimes U(\mathfrak{gl}(m|n))$ (notice this doesn't depend on the choice of $R$). 
\end{cor}
Let
\begin{align*}
    &G_1^{\mathfrak{gl}(m|n)} \coloneqq \sum_{i\in I} e_{ii},\\
    &G_{1}^{\mathfrak{gl}(m+1|n)} \coloneqq \sum_{i\in \hat{I}} e_{ii}
\end{align*}
be the {\it first Gelfand invariants} of $\mathfrak{gl}(m|n)$ and $\mathfrak{gl}(m+1|n)$, respectively (for the full definition of the Gelfand invariants see Section~\ref{sec: capelli berezinians}). Note that they are both central in the corresponding universal enveloping algebras.
Now, let us define the following homomorphisms: 
\begin{enumerate}
    \item 
    the natural embedding \[\iota^s: U(\mathfrak{sl}(m+1|n))\to U(\mathfrak{gl}(m+1|n));\] 
    \item 
    (for $m+1\ne n$) the projection \[\pi^g: U(\mathfrak{gl}(m+1|n))\to U(\mathfrak{sl}(m+1|n))\] defined by \[\pi^g(B)= B-\frac{1}{m+1-n}\mathrm{str}(B)G_{1}^{\mathfrak{gl}(m+1|n)}\]
    for $B\in \mathfrak{gl}(m|n)$ and the universal property; 
    \item 
    and the map \[\iota^g: U(\mathfrak{gl}(m|n))\to U(\mathfrak{sl}(m+1|n))\] defined by $\iota^g(C) = C-\mathrm{str}(C)e_{00}$ for $C \in \mathfrak{gl}(m|n)$ and the universal property.
\end{enumerate}
For the following proposition, let us suppose $m+1\ne n$ and set 
\begin{align}\label{R1}
    R_1 \coloneqq -\frac{1}{m-n+1}\left(\mathcal{E}\otimes 1+1\otimes G_1^{\mathfrak{gl}(m|n)}\right).
\end{align}
From the formulas we deduce that $\varphi_{R_1} = \varphi^s \circ \pi^g$ (note that $R_1$ is a super-analog of the $R_1$ defined in \cite{GR_2021}). Let us also define
\[\gamma: U(\mathfrak{sl}(m+1|n))\to \mathcal{D}'(m|n)\otimes U(\mathfrak{sl}(m+1|n))\]
by $\gamma \coloneqq (1\otimes \iota^g) \circ \varphi^s$. 
\begin{prop}\label{prop: commutative diagram}
    For $m+1\ne n$, we have $\gamma = (1 \otimes \iota^g)\varphi^s$, $\pi^g\iota^s = \mathrm{Id}$, and 
    $\varphi^s\pi^g = \varphi_{R_1}$, and all other relations that directly follow from these three; in that sense, the following diagram is commutative (in the category of superalgebras).

    \medskip
    \begin{center}
    \begin{tikzpicture}[shorten >=1pt,node distance=4cm,on grid,auto]
      \tikzstyle{every state}=[fill={rgb:black,1;white,10}]
        \node (q_0)                  {$U(\mathfrak{gl}(m+1|n))$};
        \node (q_1)  at (6,0)        {$\mathcal{D}'(m|n)\otimes U(\mathfrak{gl}(m|n))$};
        \node (q_2)  at (6,-2.5)     {$\mathcal{D}'(m|n)\otimes U(\mathfrak{sl}(m+1|n))$};
        \node (q_3)  at (0,-2.5)     {$U(\mathfrak{sl}(m+1|n))$};
        \node (p_3)  at (-0.2,-2.25) {};
        \node (p_0)  at (-0.2,-0.15) {};
        
        \path[->,thick,-{Stealth[width=5pt, length=10pt]}]
        (q_0) edge node {$\varphi_{R_1}$}             (q_1)
              edge node {$\pi^g$}            (q_3)
        (q_1) edge node {$1\otimes \iota^g$} (q_2)
        (q_3) edge node {$\varphi^s$}           (q_1)
              edge node {$\gamma$}           (q_2)
        (p_3) edge node {$\iota^s$}          (p_0);
    \end{tikzpicture}
 
    \end{center}
\end{prop}
Most of the diagram still holds if instead of $\varphi_{R_1}$ we take any $\varphi_R$ for any $R \in Z(\mathcal{D}'(m|n) \otimes U(\mathfrak{gl}(m|n)))$. The only change to the diagram is that the $\pi^g$ arrow is deleted (and the assumption $m+1 \neq n$ can be relaxed).

\section{Inflating representations}
\label{sec: inflating}
In this section we use $\varphi$ (recall that $\varphi$ is shorthand for $\varphi_R$) to extend a simple representation of $U(\mathfrak{gl}(m|n))$ to a representation of $U(\mathfrak{gl}(m+1|n))$ and find conditions for when this new representation is simple.

For any $a \in \mathbb{C}$, define the $\mathcal{D}'(m|n)$-module 
\[\mathcal{F}_a \coloneqq \mathrm{span}\{t_0^{a-k_1-\cdots - k_{\bar{n}}}t_1^{k_1}\cdots t_n^{k_m} t_{\bar{1}}^{k_{\bar{1}}}\cdots t_{\bar{n}}^{k_{\bar{n}}}\mid k_1,\ldots, k_m\in \mathbb{Z}_{\ge 0}, \ k_{\bar{1}},\ldots, k_{\bar{n}}\in\{0,1\}\}.\]
This is a generalization of the $\mathcal{F}_a$ defined in \cite{GR_2021}, both in the sense that it is defined for superalgebras and in that we do not require $a\in \mathbb{Z}$.
Note that $\mathcal{F}_a = \mathcal{D}'(m|n)(t_0^a)$ (i.e. the set that results when we apply elements of $\mathcal{D}'(m|n)$ to $t_0^a$) and that $\mathcal{E} = a\mathrm{Id}$ on $\mathcal{F}_a$. Given a weight $\lambda$ of $\mathfrak{gl}(m|n)$, we write $\mathcal{F}_a\otimes L(\lambda)$ (and $\mathcal{F}_a\otimes M(\lambda)$) for both the representation of $\mathcal{D}'(m|n)\otimes U(\mathfrak{gl}(m|n))$ and the representation  $U(\mathfrak{gl}(m+1|n))$ via $\varphi$, and the meaning should be clear from context. We fix a highest weight vector $v_\lambda$ of $L(\lambda)$.

Let us now consider inflating the simple module $L(\lambda)$ using $\varphi$. In particular, we investigate when the resulting module is simple.
\begin{prop}\label{prop: only highest weight vector}
    Let $v_\lambda$ be the unique highest weight vector for $L(\lambda)$, up to scaling. The vector $t_0^{a}\otimes v_\lambda$ is the \textit{only} highest weight vector (up to scaling) for $\mathcal{F}_a\otimes L(\lambda)$, considered as a $U(\mathfrak{gl}(m+1|n))$-module.
\end{prop}
\begin{proof}
For $0<i<j$, we have that $e_{ij}$ acts as $t_i\partial_j\otimes 1 + 1\otimes e_{ij}$. Now $t_i\partial_j(t_0^a) = 0$ and $e_{ij}$ acts on $v_\lambda$ as $0$ since $i<j$, so $e_{ij}$ acts on $t_0^a\otimes v_\lambda$ as $0$. Also, for each $j > 0$, $e_{0j}$ acts on $t_0^a\otimes v_\lambda$ as $t_0\partial_j\otimes 1$. Now $t_0\partial_j (t_0^a)=0$, so $e_{0j}$ also acts as $0$. It follows that $t_0^a\otimes v_\lambda$ is a highest weight vector.

Now suppose
\begin{align}
    \widetilde{w} = p_1\otimes w_1+\cdots p_k\otimes w_k
\end{align}
is a highest weight vector, where \[p_c=t_0^{a-k_{1,c}-k_{2,c}-\cdots - k_{m,c}-k_{\bar{1},c}-\cdots - k_{\bar{n},c}}t_1^{k_{1,c}}\cdots t_n^{k_{m,c}} t_{\bar{1}}^{k_{\bar{1},c}}\cdots t_{\bar{n}}^{k_{\bar{n},c}},\] where $k_{1,c},\ldots, k_{m,c}\in \mathbb{Z}_{\ge 0}$ and $k_{\bar{1},c},\ldots, k_{\bar{n},c}\in\{0,1\}$. Furthermore assume the $p_c$ are distinct and the $w_c$ are nonzero. 
Now $e_{0j}$ (for $j\in I$), which acts as $t_0\partial_j\otimes 1$, must act as zero on $\widetilde{w}$. Pick the $p_c$ that is $t$-lexicographically maximal (i.e. relative to the $t_1$-degree, $t_2$-degree, and so on). Then,  $e_{01}^{k_{1,c}}\cdots e_{0\bar{n}}^{k_{\bar{n},c}}$ acts as zero on all $p_d\otimes w_d$ except for $p_c\otimes w_c$, on which it gives a nonzero result. Thus $e_{01}^{k_{1,c}}\cdots e_{0\bar{n}}^{k_{\bar{n},c}}$ does not act as zero on $\widetilde{w}$, a contradiction unless $p_c = t_0^a$, in which case $k=1$ and we can write $\widetilde{w} = t_0^a\otimes w$.

Then, for $i<j$, $e_{ij}$ acts as $t_i\partial_j\otimes 1 + 1 \otimes e_{ij}$, and this must act as zero on $\widetilde{w}$. First, note that the $t_i\partial_j\otimes 1$ part acts as zero. Thus we must have $e_{ij}$ acting on $w$ as $0$ for all $i<j$, i.e. $w$ is a highest weight vector of $L(\lambda)$, so $w$ is a scalar multiple of $v_\lambda$. Thus, $t_0^a \otimes v_\lambda$ is the only the highest weight vector (up to scaling) of $\mathcal{F}_a \otimes L(\lambda)$, considered as a $U(\mathfrak{gl}(m+1|n))$-module.
\end{proof}

We can write $R$ in the form $R = \sum_{i} \mathcal{E}^i \otimes z_i$ for some $z_i \in Z(U(\mathfrak{gl}(m|n)))$ by Lemma \ref{lem: identities for D'(m|n)}. A direct computation shows that $R$ acts as a scalar on $\mathcal{F}_a \otimes L(\lambda)$, which we will call $r$:

\begin{align*}
R\cdot\left(\sum_j f_j \otimes v_j\right) &= \left(\sum_{i} \mathcal{E}^i \otimes z_i\right)\left(\sum_j f_j \otimes v_j\right) = \sum_{i,j} (\deg f_j)^i f_j \otimes \chi_\lambda(z_i)v_j \\
&= \left(\sum_i a^i \chi_\lambda(z_i) \right)\left(\sum_{j} f_j \otimes v_j\right) = r\left(\sum_{j} f_j \otimes v_j\right),
\end{align*}
where $f_j \in \mathcal{F}_a, v_j \in L(\lambda)$ are arbitrary, and $\chi_\lambda$ is the central character corresponding to $\lambda$ (see Section~\ref{sec: harish chandra}). It follows that the weight for $t_0^a\otimes v_\lambda$ is
\begin{align}\label{weight}
\widetilde{\lambda} \coloneqq (a + r, \lambda_1 + r, \ldots, \lambda_m + r, \lambda_{\bar{1}} - r, \ldots, \lambda_{\bar{n}} -r).
\end{align}
Next, we find conditions on $\lambda$ for $t_0^a\otimes v_\lambda$ to generate the module $\mathcal{F}_a\otimes L(\lambda)$. For $i\in I$, define
\[f(i) \coloneqq\sum_{j<i} (-1)^{|j|},\]
where the sum runs over $j\in I$ with $j<i$ (recall the ordering on $I$ defined by $1<2<\cdots < m< \bar{1}<\cdots < \bar{n}$). Then $f(1)=0, f(2)=1,\ldots, f(m)=m-1, f(\bar{1}) = m, f(\bar{2})= m-1,\ldots, f(\bar{n}) = m+1-n$.

\begin{thm}\label{thm: when the highest weight vector generates the module}
    Let $\lambda$ be a $\mathfrak{gl}(m|n)$-weight. If $a+f(i)-(-1)^{|i|}\lambda_i$ is not a nonnegative integer for all $i \in I$, then $\mathcal{F}_a \otimes L(\lambda)$ is a highest weight module.
\end{thm}

\begin{proof}
Let $N$ be the submodule generated by the highest weight vector $t_0^a\otimes v_\lambda$. We first claim that all monomials in $\mathcal{F}_a$, and thus all elements of $\mathcal{F}_a$, create an element of $N$ when tensored with $v_\lambda$, the highest weight vector for $L(\lambda)$. We proceed by induction down on the exponent of $t_0$. Specifically, assume that $T \otimes v_\lambda= t_0^{a_0}S \otimes v_\lambda = t_0^{a_0}t_1^{a_1}\cdots t_{\bar{n}}^{a_{\bar{n}}}
\otimes v_\lambda \in N$ for $a_0 + a_1+\cdots  + a_{\bar{n}}= a$, and $a_i\in \mathbb{N}$ for $|i|=0$ and $a_i\in\{0,1\}$ for $|i|=1$. We claim that $(t_i/t_0 \otimes 1)(T \otimes v_\lambda) \in N$ for any $i\in I$. To do this we induct on $i$.
First look at the case $i = 1$. We have
\[e_{10}\cdot (t_0^{a_0}S\otimes v_\lambda) = (t_1\partial_0 \otimes 1 - \sum_{j}\frac{t_j}{t_0} \otimes e_{1j})(t_0^{a_0}S \otimes v_\lambda) = a_0t_1t_0^{a_0-1}S \otimes v_\lambda- t_1t_0^{a_0-1}S \otimes \lambda_1 v_\lambda.\]
which means $(t_1/t_0\otimes 1)(T\otimes v_\lambda)$ is in $N$ as long as $a_0 - \lambda_1\ne 0$. But since $a-\lambda_1$ is not a nonnegative integer, and $a-a_0$ is a nonnegative integer, $a_0-\lambda_1\ne 0$.
Thus the base case is proved.
For general $i \ne 0$, note that
\begin{align*}
    e_{i0} \cdot (t_0^{b}S \otimes v_\lambda) &= \left(t_i\partial_0 \otimes 1 - \sum_j (-1)^{|i||j|}\frac{t_j}{t_0}\otimes e_{ij}\right)(t_0^{b}S \otimes v_\lambda)\\ &= bt_it_0^{b-1}S \otimes v_\lambda - (-1)^{|i|}\lambda_it_it_0^{b-1}S \otimes v_\lambda\\ &\phantom{=}- \sum_{j < i} (-1)^{|i||j| + (|i|+|j|)(|t_0^bS|)}t_jt_0^{b-1}S \otimes (e_{ij} \cdot v_\lambda).
\end{align*}
In addition, for $0<j<i$, if $|j| = 0$, we have
\begin{align*}
    e_{ij} \cdot (t_jt_0^{b-1}S \otimes v_\lambda) &= (t_i\partial_j \otimes 1)(t_jt_0^{b-1}S \otimes v_\lambda) + (1\otimes e_{ij})(t_jt_0^{b-1}S \otimes v_\lambda)\\
    &=(1+a_j) t_it_0^{b-1}S \otimes v_\lambda + (-1)^{|i||j| + |j| + (|i|+|j|)(|t_0^bS|)} (t_jt_0^{b-1}S \otimes (e_{ij}\cdot v_\lambda))\\
    &= (1+a_j) t_it_0^{b-1}S \otimes v_\lambda + (-1)^{|i||j| + (|i|+|j|)(|t_0^bS|)} (t_jt_0^{b-1}S \otimes (e_{ij}\cdot v_\lambda)).
\end{align*}
If $|j| = 1$, then
\begin{align*}
    -e_{ij} \cdot (t_jt_0^{b-1}S \otimes v_\lambda) &= -(t_i\partial_j \otimes 1)(t_jt_0^{b-1}S \otimes v_\lambda) - (1\otimes e_{ij})(t_jt_0^{b-1}S \otimes v_\lambda)\\
    &=-(1-a_j) t_it_0^{b-1}S \otimes v_\lambda - (-1)^{|i||j| + |j| + (|i|+|j|)(|t_0^bS|)} (t_jt_0^{b-1}S \otimes (e_{ij}\cdot v_\lambda))\\
    &= -(1-a_j) t_it_0^{b-1}S \otimes v_\lambda + (-1)^{|i||j| + (|i|+|j|)(|t_0^bS|)} (t_jt_0^{b-1}S \otimes (e_{ij}\cdot v_\lambda)).
\end{align*}
since if $a_j = 0$, then $(t_i\partial_j)(t_jt_0^{b-1}S) = t_it_0^{b-1}$ and if $a_j=1$, then $t_jt_0^{b-1}S=0$ since $t_j^2 = 0$.
Also, both of these are in $N$ because $t_jt_0^{b-1}S\otimes v_\lambda\in N$, from the inductive hypothesis.
Adding these equations (for all $j<i$) to the first gives us that
\[\left(b+\sum_{\substack{j<i\\ |j| = 0}}(1+a_j) + \sum_{\substack{j<i\\ |j| = 1}} (a_j-1) - (-1)^{|i|}\lambda_i\right)t_it_0^{b-1}S\otimes v_\lambda \in N.\]

Now we show that the expression in parentheses (call it $A$) is nonzero. Note
\[A = b+\sum_{j<i}a_j + f(i)-(-1)^{|i|}\lambda_i =  b+\sum_{j\in I}a_j +f(i) - (-1)^{|i|}\lambda_i- K  = a+f(i) - (-1)^{|i|}\lambda_i- K,\]
where $K$ is some nonnegative integer. Since $a+f(i)-(-1)^{|i|}\lambda_i$ is not a nonnegative integer, it follows that $A \ne 0$
so $t_it_0^{b-1}S\otimes v_\lambda\in N$.

Now, we show that $p\otimes (e_{ij}\cdot v_\lambda)\in N$, for any $p\in \mathcal{F}_a$ and $i,j\in I$. We have
\[e_{ij}\cdot (p\otimes v_\lambda) = (t_i\partial_j)(p)\otimes v_\lambda + p\otimes (e_{ij}\cdot v_\lambda) + \delta_{ij}(-1)^{|i||j|}R(p\otimes v_\lambda),\]
and this element is in $N$. Now $R$ acts as $r \mathrm{Id}$ (see the discussion preceding (\ref{weight})), so $R(p\otimes v_\lambda)\in N$ and $(t_i\partial_j)(p)\otimes v_\lambda\in N$ from the first part of the proof. Thus $p\otimes (e_{ij}\cdot v_\lambda) \in N$ for all $p\in \mathcal{F}_a$ and $i,j\in I$.
Repeating this process multiple times shows that $p\otimes (e_{i_1 j_1}e_{i_2j_2}\cdots e_{i_kj_k}\cdot v_\lambda)\in N$ for any $k$. Since $v_\lambda$ generates $L(\lambda)$, this shows $p\otimes w\in N$ for any $p\in \mathcal{F}_a$ and $w\in L(\lambda)$. Thus $N = \mathcal{F}_a \otimes L(\lambda)$.
\end{proof}
Using the following theorem, we can use the previous theorem to find conditions when $\mathcal{F}_a\otimes L(\lambda)\cong L(\widetilde{\lambda})$.
\begin{thm}\label{thm: inflate}
    For a weight $\lambda$ of $\mathfrak{gl}(m|n)$, we have $\mathcal{F}_a\otimes L(\lambda)\cong L(\widetilde{\lambda})$ as $\mathfrak{gl}(m+1|n)$-modules if and only if $\mathcal{F}_a\otimes L(\lambda)$ is generated by the highest weight vector $t_0^a \otimes v_\lambda$.
\end{thm}
\begin{proof}
    First note that by Proposition \ref{prop: only highest weight vector}, showing $\mathcal{F}_a\otimes L(\lambda)\cong L(\widetilde{\lambda})$ is equivalent to showing $\mathcal{F}_a\otimes L(\lambda)$ is simple. Let $N$ denote the submodule of $\mathcal{F}_a \otimes L(\lambda)$ generated by $t_0^a \otimes v_\lambda$. If $N\ne \mathcal{F}_a\otimes L(\lambda)$, it is obvious that $\mathcal{F}_a\otimes L(\lambda)$ cannot be simple. Now we prove the ``if'' direction.

    Since $N=\mathcal{F}_a\otimes L(\lambda)$, we know $t_0^a\otimes v_\lambda$ generates $\mathcal{F}_a\otimes L(\lambda)$. Let $x$ be a nonzero element of $\mathcal{F}_a\otimes L(\lambda)$. Let $N(x)$ be the module generated by $x$. We will show that $t_0^a\otimes v_\lambda\in N(x)$, so $N(x)=\mathcal{F}_a\otimes L(\lambda)$, thereby showing $\mathcal{F}_a\otimes L(\lambda)$ is simple, since $x$ is an arbitrary element.

    Write
    \begin{align}\label{eqn: pure tensor decomposition}
        x = c_1p_1\otimes w_1+\cdots c_kp_k\otimes w_k,
    \end{align}
    where \[p_c=t_0^{a-k_{1,c}-k_{2,c}-\cdots - k_{m,c}-k_{\bar{1},c}-\cdots - k_{\bar{n},c}}t_1^{k_{1,c}}\cdots t_n^{k_{m,c}} t_{\bar{1}}^{k_{\bar{1},c}}\cdots t_{\bar{n}}^{k_{\bar{n},c}},\] where $k_{1,c},\ldots, k_{m,c}\in \mathbb{Z}_{\ge 0}$ and $k_{\bar{1},c},\ldots, k_{\bar{n},c}\in\{0,1\}$, $w_c = e_{i_{1,c}j_{1,c}}e_{i_{2,c}j_{2,c}}\cdots e_{i_{d,c}j_{d,c}}\cdot v_\lambda$, and finally $c_1,\ldots c_k\in \mathbb{C}_{\ne 0}$. We claim that a nonzero element of the form 
    \begin{align}\label{eqn: nice element}
        x' = c_1't_0^a\otimes w_1'+\cdots c_{k'}'t_0^a\otimes w_{k'}'
    \end{align}
    exists in $N(x)$. For $x$ in the form (\ref{eqn: pure tensor decomposition}), we define
    \begin{align}\label{eqn: sum}
    S(x) =\sum_{i\in I}\sum_{\ell=1}^k k_{i,\ell}.
    \end{align}
    Note $S(x)$ is defined uniquely for all $x$ and $S(x)$ is a nonnegative integer. Additionally, $S(x')=0$ if and only if $x'$ is in the form (\ref{eqn: nice element}). If $S(x) = 0$, we are done.
    Otherwise, there exists $k_{j,\ell}>0$ for some $j,\ell$. Applying $e_{0j}$ to $x$ (where we have $e_{0j}\cdot x = (t_0\partial_j)x$), we see that $S(e_{0j}\cdot x)< S(x)$, so $S(\cdot)$ decreases under this action. Also, note that $e_{0j}\cdot x$ is nonzero since $k_{j,\ell}>0$. If $S(e_{0j}\cdot x)=0$, we are done. Otherwise we can do the same process to $e_{0j}\cdot x$, finding $e_{0j'} e_{0j}\cdot x\ne 0$ with $S(e_{0j'} e_{0j}\cdot x)< S(e_{0j}\cdot x)$. If $S(e_{0j'} e_{0j}\cdot x)=0$ we are done. Otherwise we can just keep continuing this process, which must terminate since $S$ is always an integer.    
    Then it follows that we can find $x' \in N(x)$ which can be written as~(\ref{eqn: nice element}). We can then write $x'$ as
    \[x' = t_0^a\otimes w,\]
    where $w\in \mathcal{F}_a\otimes L(\lambda)$ and $w\ne 0$.

    Next, we show that if $t_0^a\otimes w'\in N(x)$, then $t_0^a\otimes (e_{ij}\cdot w')\in N(x)$, for $i,j\in I$. We have
    \begin{align*}
        e_{ij}\cdot (t_0^a\otimes w') &= (t_i\partial_j)\cdot (t_0^a\otimes w') + (1\otimes e_{ij})\cdot (t_0^a\otimes w') + \delta_{ij}(-1)^{|i||j|}R\cdot (t_0^a\otimes w')\\
        &= \pm t_0^a\otimes (e_{ij}\cdot w') + s(t_0^a\otimes w')
    \end{align*}
    for a constant $s\in \mathbb{C}$. Subtracting $s(t_0^a\otimes w')$, we get that $t_0^a\otimes (e_{ij}\cdot w')\in N(x)$.

    Let $N'(w)$ be the submodule generated by $w$ in $L(\lambda)$.
    Since $x' = t_0^a\otimes w\in N(x)$, from the proceeding argument, it follows that $t_0^a\otimes N'(w)\subset N(x)$. But since $L(\lambda)$ is simple, it follows that $N'(w)=L(\lambda)$. In particular, $v\in N'(w)$, so $t_0^a\otimes v_\lambda\in N(x)$. Thus $N(x)=\mathcal{F}_a\otimes L(\lambda)$, so $\mathcal{F}_a\otimes L(\lambda)$ is simple.
\end{proof}

\begin{cor}\label{cor: combine thms}
    Let $\lambda$ be a $\mathfrak{gl}(m|n)$-weight. If $(\widetilde{\lambda}+\rho_{m+1|n}, \delta_0 - \delta_i)\not\in \mathbb{Z}_{>0}$ for $1\le i\le m$ and $(\widetilde{\lambda}+\rho_{m+1|n}, \delta_0 - \delta_i)\not\in \mathbb{Z}_{\ge 0}$ for $\bar{1}\le i \le \bar{n}$, then $\mathcal{F}_a\otimes L(\lambda)$ is simple, which implies $\mathcal{F}_a\otimes L(\lambda)\cong L(\widetilde{\lambda})$. 
\end{cor}

\begin{proof}
    We show that the conditions given are equivalent to the conditions in Theorem~\ref{thm: when the highest weight vector generates the module}. Letting $\rho = \rho_{m|n}$, note that $f(i)+1 = \rho_0 - \rho_i$ for $1\le i\le m$ and $f(i) = \rho_0 + \rho_i$ for $\bar{1}\le i\le \bar{n}$, where $\rho = \rho_{m+1|n}$. It follows that the condition $a+f(i)-(-1)^{|i|}\lambda_i$ not being a nonnegative integer is equivalent to the conditions stated in the corollary. Then the result follows by Theorem \ref{thm: inflate} and Theorem \ref{thm: when the highest weight vector generates the module}.
\end{proof}

Let us recall the following fact: if $\lambda$ is both antidominant and typical, then $M(\lambda)$ is simple, see for reference \cite{Gorelik_2001a} or \cite{Gorelik_2001b}. Our result will be a consequence of Corollary \ref{cor: combine thms}.

\begin{cor}\label{cor: verma}
    Let $\mathcal{AW}_{m|n}(a)$ denote the set of all typical antidominant $\mathfrak{gl}(m|n)$-weights $\lambda$ such that $\lambda_{i}-(-1)^{|i|}a\not\in \mathbb{Z}$ for $i\in I$. If $\lambda \in \mathcal{AW}_{m|n}(a)$, then $\mathcal{F}_a\otimes M(\lambda)$, considered as a $\mathfrak{gl}(m+1|n)$-representation, is isomorphic to $M(\widetilde{\lambda})$.
\end{cor} 
\begin{proof}
    Since $\lambda$ is antidominant and typical, $L(\lambda)=M(\lambda)$. 
    For $1\le i<j\le m$ or $\bar{1}\le i<j\le \bar{n}$, we have $(\widetilde{\lambda}+\rho_{m+1|n})_i - (\widetilde{\lambda}+\rho_{m+1|n})_j = (\lambda+\rho_{m|n})_i-(\lambda+\rho_{m|n})_j\not\in \mathbb{Z}_{>0}$. Additionally, note that for $1\le i\le m$, $(\widetilde{\lambda} + \rho_{m+1|n})_0 - (\widetilde{\lambda} + \rho_{m+1|n})_i = a + i - \lambda_i \not\in \mathbb{Z}_{>0}$ since $a-\lambda_i \not\in \mathbb{Z}$. Thus $\widetilde{\lambda}$ is antidominant. 

    For $1\le i\le m$ and $\bar{1}\le j\le \bar{n}$, we have 
    \[(\widetilde{\lambda} + \rho_{m|n},\delta_i-\delta_j) =(\widetilde{\lambda}+\rho_{m+1|n})_i + (\widetilde{\lambda}+\rho_{m+1|n})_j = (\lambda+\rho_{m|n})_i + (\lambda+\rho_{m|n})_j \ne 0\]
    since $\lambda$ is typical. Finally, for $\bar{1}\le j\le \bar{n}$, we have
    \[(\widetilde{\lambda} + \rho_{m|n},\delta_0-\delta_j) =(\widetilde{\lambda}+\rho_{m+1|n})_0 + (\widetilde{\lambda}+\rho_{m+1|n})_j = a+m+1 + \lambda_j -j \ne 0\]
    because $a+\lambda_j$ is not an integer. Therefore $\widetilde{\lambda}$ is typical as well as antidominant.
   
    Thus $M(\widetilde{\lambda})=L(\widetilde{\lambda})$ is simple. Then note that $\lambda$ satisfies the conditions in Corollary \ref{cor: combine thms}, so the result follows.
\end{proof}
Now suppose $\xi\in \mathfrak{h}_{m+1|n}^*$ such that $(\xi+\rho_{m+1|n}, \alpha^\vee) \not\in \mathbb{Z}_{> 0}$ for $\alpha \in \Phi_{\overline{0}}^+$ and $(\xi+\rho_{m+1|n}, \alpha)\not\in \mathbb{Z}_{\ge 0}$ for $\alpha \in \Phi_{\overline{1}}^+$. Note that the first condition is simply that $\xi$ is antidominant and the second implies $\xi$ is typical. Then inductively apply Corollary~\ref{cor: combine thms} to use the modules $\mathcal{F}_a$ to ``build'' the representation $L(\xi)$ for $\xi \in \mathfrak{h}_{m+1|n}^*$. The conditions on $\xi$ allow us to apply Corollary \ref{cor: combine thms} iteratively (with $R=0$) to get a map
\begin{align*}
    U(\mathfrak{gl}(m+1|n)) \rightarrow \mathcal{D}'(m|n)\otimes \cdots \otimes \mathcal{D}'(0|n)\otimes \mathcal{D}'(n-1|0) \otimes \cdots \otimes \mathcal{D}'(0|0)
\end{align*}
that gives us a $\mathfrak{gl}(m+1|n)$-representation
\[\mathscr{F}(\xi) \coloneqq \mathcal{F}_{\xi_0}\otimes \mathcal{F}_{\xi_1}\otimes \cdots \otimes \mathcal{F}_{\xi_{\bar{n}}}\]
that is isomorphic to $L(\xi)$.
(Note that when constructing the chain have to use the obvious isomorphism from $U(\mathfrak{gl}(0|n))$ to $U(\mathfrak{gl}(n|0))$ given by $e_{\bar{i}\bar{j}}\mapsto e_{ij}$).  
But $\xi$ is antidominant and typical, so $\mathscr{F}(\xi)$ is isomorphic to $M(\xi)$ as well. This applies to substantially more general setting than \ref{cor: verma}, and motivates the following conjecture:
\begin{conj}
There is an isomorphism of $U(\mathfrak{gl}(m+1|n))$-modules $\mathscr{F}(\xi)\cong M(\xi)$ for \textit{all} weights $\xi \in \mathfrak{h}_{m+1|n}^*$.
\end{conj}
Further supporting evidence is that the modules have the same formal supercharacters, which is verified in a manner analogous to the proof of Theorem 4.2, page 6 in \cite{GR_2021}.

\section{Images under Harish-Chandra homomorphisms}\label{sec: harish chandra}
In this section, we relate the restriction of $\varphi$ to the center of $U(\mathfrak{gl}(m+1|n))$ with the Harish-Chandra homomorphisms.

For $M\in \{m+1,m\}$, recall $\rho_{M|n} = (M,\ldots, 1, -1, \ldots, -n)$. If $b=(b_1,\ldots, b_M, b_{\bar{1}},\ldots, b_{\bar{n}})\in \mathbb{C}^{M|n}$, then the evaluation homomorphism $\mathrm{ev}_{b}: \mathbb{C}[l_1,\ldots, l_M,l_{\bar{1}},\ldots, l_{\bar{n}}]\to \mathbb{C}$ is defined by \[\mathrm{ev}_b(p) \coloneqq p(b_1,\ldots, b_M,b_{\bar{1}},\ldots, b_{\bar{n}}).\]

Every $z\in Z(U(\mathfrak{gl}(m|n))$ acts on $L(\lambda)$ as $\chi_{\lambda}(z)\mathrm{Id}$, where $\chi_\lambda(z) = \mathrm{ev}_{\lambda+\rho_{m|n}}(\chi_{m|n}(z))$ and $\chi_{m|n}: Z(U(\mathfrak{gl}(m|n)))\to \mathbb{C}[\ell_1,\ldots,\ell_m, \ell_{\bar{1}},\ldots, \ell_{\bar{n}}]^{\mathcal{W}_{m|n}}$ is the Harish-Chandra homomorphism (see Theorem 13.1.1 (a) in \cite{Musson_2012}). We similarly define $\chi_{m+1|n}: Z(U(\mathfrak{gl}(m+1|n)))\to \mathbb{C}[\ell_0,\ldots,\ell_m, \ell_{\bar{1}},\ldots, \ell_{\bar{n}}]^{\mathcal{W}_{m+1|n}}$.

Next, we define $\chi_{0,m|n}: \mathbb{C}[\mathcal{E}] \otimes Z(U(\mathfrak{gl}(m|n)))\to \mathbb{C}[\ell_0] \otimes \mathbb{C}[\ell_1,\ldots, \ell_{\bar{n}}]^{\mathcal{W}_{m|n}}$ by \[\chi_{0,m|n}\left(\sum_{i}\mathcal{E}^i\otimes z_i\right) = \sum_i \ell_0^i\otimes \chi_{m|n}(z_i).\] Set $\ell = \chi_{0,m|n}(R)$. In the case when $R$ is given by (\ref{R1}), we have
\[\ell = -\frac{1}{n+1}\left( \sum_{i \in \hat{I}}\ell_i - \frac{m(m+1)}{2} + \frac{n(n+1)}{2}\right).\]
 
Note that the symmetric superalgebra $S(\mathfrak{h}_{m|n})$ is isomorphic, in a canonical way, to the superalgebra of $\mathcal{W}_{m|n}$-invariant polynomials in $\mathbb{C}[\ell_1, \ldots, \ell_{\bar{n}}]$. Let $I(\mathfrak{h}_{m|n})$ be the subalgebra of $S(\mathfrak{h}_{m|n})$ consisting of all $\mathcal{W}_{m|n}$-invariant functions $\theta$ on $\mathfrak{h}_{m|n}^*$ such that if $\alpha$ is an isotropic odd root and $(\lambda, \alpha)=0$, then $\theta(\lambda)=\theta(\lambda + t\alpha)$ for all $t\in \mathbb{C}$. It is known that $\chi_{m|n}: Z(U(\mathfrak{gl}(m|n)))\to I(\mathfrak{h}_{m|n}) $ is an isomorphism (see Theorem 13.1.1 (b) in \cite{Musson_2012}). Note that this means $\chi_{0, m|n}$, considered as a map from $\mathbb{C}[\mathcal{E}]\otimes Z(U(\mathfrak{gl}(m|n)))$ to $\mathbb{C}[\ell_0]\otimes I(\mathfrak{h}_{m|n}) ,$ is an isomorphism as well.

For the remainder of this section, we consider the module $\mathcal{F}_a$ only when $a$ is an integer. Then note $\mathcal{AW}_{m|n}\coloneqq\mathcal{AW}_{m|n}(a) = \mathcal{AW}_{m|n}(0)$.
\begin{lem}\label{lem: faithful}
    The modules \[\mathcal{F} = \bigoplus_{a\in \mathbb{Z}}\mathcal{F}_a \text{ and } \bigoplus_{\lambda \in \mathcal{AW}_{m|n}}M(\lambda)\] are faithful over $\mathcal{D}'(m|n)$ and $\mathfrak{gl}(m|n)$, respectively.
\end{lem}
\begin{proof}
    We prove that both modules have trivial annihilators.

    If $z\in Z(U(\mathfrak{gl}(m|n)))$ annihilates $\bigoplus_{\lambda\in \mathcal{AW}_{m|n}}M(\lambda)$, then $\chi_{\lambda}(z)=0$ for all $\lambda \in \mathcal{AW}_{m|n}$. Since $\chi_{\lambda}(z) = \mathrm{ev}_{\lambda + \rho_{m|n}}(\chi_{m|n}(z))$, we have
    \[\mathrm{ev}_{\lambda+\rho_{m|n}}(\chi_{m|n}(z))=0\]
    for all $\lambda\in \mathcal{AW}_{m|n}$. But this can only happen if $\chi_{m|n}(z)=0$. Thus, we must have $z = 0$.

    Now suppose $x\in \mathcal{D}'(m|n)$ annihilates $\mathcal{F}$ and $x$ is nonzero. Write $x$ as the sum of elements of the form
    \[
    p(t_0,\ldots, t_{\bar{n}})\cdot \partial_0^{b_0}\cdots\partial_{\bar{n}}^{b_{\bar{n}}} 
    \]
    where $p(t_0,\ldots, t_{\bar{n}})$ is some nonzero polynomial. Furthermore, assume that each term in the sum has a different $(b_{0},\ldots, b_{\bar{n}})$. Pick the term that is $\partial$-lexicographically maximal (i.e. relative to the $\partial_0$-degree, the $\partial_1$-degree, and so on). Consider the polynomial $t_0^{b_0}t_1^{b_1}\cdots t_{\bar{n}}^{b_{\bar{n}}}$. Then the leximaximal property of the term we picked guarantees that all the other terms act as zero. Furthermore,
    \begin{align*}
        (\partial_0^{b_0}\cdots\partial_{\bar{n}}^{b_{\bar{n}}})(t_0^{b_0}\cdots t_{\bar{n}}^{b_{\bar{n}}}) = c
    \end{align*}
    for a nonzero $c\in \mathbb{R}$. Therefore
    \begin{align*}
        x\cdot (t_0^{b_0}\cdots t_{\bar{n}}^{b_{\bar{n}}}) &= (p(t_0,\ldots, t_{\bar{n}})\cdot \partial_0^{b_0}\cdots\partial_{\bar{n}}^{b_{\bar{n}}} )(t_0^{b_0}\cdots t_{\bar{n}}^{b_{\bar{n}}})\\
        &= c\cdot p(t_0,\ldots, t_{\bar{n}})\ne 0.
    \end{align*}
    This contradicts the fact that $x$ annihilates $\mathcal{F}$. Therefore $\mathcal{F}$ has a trivial annihilator. 
\end{proof}

Define
\[\tau(p(\ell_0,\ldots, \ell_{m},\ell_{\bar{1}},\ldots \ell_{\bar{{n}}})) \coloneqq p(\ell_0 + \ell + m+1, \ell_1 + \ell, \ldots, \ell_{m}+\ell, \ell_{\bar{{1}}} - \ell, \ldots, \ell_{\bar{{n}}}-\ell).\]
Then it is easy to see that $\tau$ is a homomorphism from $I(\mathfrak{h}_{m+1|n})$ to $\mathbb{C}[\ell_0]\otimes I(\mathfrak{h}_{m|n}) $.
\begin{thm}\label{thm : HC homomorphism}
    We have that
    \[\varphi|_{Z(U(\mathfrak{gl}(m+1|n)))} = \chi_{0,m|n}^{-1}\tau\chi_{m+1|n}.\]
    In particular, $\varphi(Z(U(\mathfrak{gl}(m+1|n))))$ is a subalgebra of $\mathbb{C}[\mathcal{E}]\otimes Z(U(\mathfrak{gl}(m|n)))$, and the following diagram is commutative:
    \begin{center}
    \begin{tikzpicture}[shorten >=1pt,node distance=4cm,on grid,auto]
      \tikzstyle{every state}=[fill={rgb:black,1;white,10}]
        \node (q_0)                  {$Z(U(\mathfrak{gl}(m+1|n)))$};
        \node (q_1)  at (6,0)        {$\mathbb{C}[\mathcal{E}]\otimes Z(U(\mathfrak{gl}(m|n)))$};
        \node (q_2)  at (6,-2.5)     {$\mathbb{C}[\ell_0]\otimes I(\mathfrak{h}_{m|n}) $};
        \node (q_3)  at (0,-2.5)     {$I(\mathfrak{h}_{m+1|n})$};
        \node (p_3)  at (-0.2,-2.25) {};
        \node (p_0)  at (-0.2,-0.15) {};
        
        \path[->,thick,-{Stealth[width=5pt, length=10pt]}]
        (q_0) edge node {$\varphi$}             (q_1)
              edge node {$\chi_{m+1|n}$}       (q_3)
        (q_1) edge node {$\chi_{0,m|n}$}   (q_2)
         (q_3)     edge node {$\tau$}           (q_2);
    \end{tikzpicture}
    \end{center}
\end{thm}
\begin{proof}
    Suppose $\lambda\in \mathcal{AW}_{m|n}$ and $a\in \mathbb{Z}$. Let $M(a,\lambda) = \mathcal{F}_a\otimes M(\lambda)$, considered as a $\mathfrak{gl}(m+1|n)$-module. We show that $\varphi(z) = \chi_{0,m|n}^{-1}\tau \chi_{m+1|n}(z)$ for $z\in Z(U(\mathfrak{gl}(m+1|n)))$, as an identity of endomorphisms of $M(a,\lambda)$. By Lemma~\ref{lem: faithful}, and by the fact that the tensor product of faithful modules is faithful, the module $\bigoplus_{a\in \mathbb{Z}, \lambda\in \mathcal{AW}_{m|n}} M(a,\lambda)$ is faithful over $\mathcal{D}'(m|n)\otimes U(\mathfrak{gl}(m|n))$. Then it will follow that $\varphi = \chi_{0,m|n}^{-1} \tau \chi_{m+1|n}(z)$. 
    
    By Corollary \ref{cor: verma}, we have $M(a,\lambda)\cong M(\widetilde{\lambda})$. Thus, every $z\in Z(U(\mathfrak{gl}(m+1|n)))$ acts on $M(a,\lambda)$ as $\chi_{\widetilde{\lambda}}(z)\mathrm{Id}$. Now set $\xi = \chi_{0, m|n}^{-1}\tau\chi_{m+1|n}$. Then, $\xi(z)$ acts on $M(a,\lambda)$ as $\chi_{a, \lambda}(\xi(z))\mathrm{Id}$, where $\chi_{a, \lambda}(\sum \mathcal{E}^i \otimes z'_i) \coloneqq \sum a^i \chi_{\lambda}(z'_i)$, because $\mathcal{E}^i \otimes 1$ acts on $M(a,\lambda)$ as $a^i$. Thus we just need to show $\chi_{a,\lambda}(\xi(z)) =\chi_{\widetilde{\lambda}}(z)$
    
    For a polynomial $p(t_0,\ldots, t_{\bar{n}})$, we have
    \begin{align*}
        &\mathrm{ev}_{a,\lambda+\rho_{m|n}} \tau(p)\\ = &\mathrm{ev}_{a,\lambda+\rho_{m|n}} p(\ell_0 + \ell + m+1, \ell_1 + \ell, \cdots, \ell_m + \ell, \ell_{\bar{1}}-\ell, \ldots, \ell_{\bar{n}}-\ell)\\
        =&p(a+ r+m+1, \ell_1 + r + m,\ldots, \ell_m + r +1, \ell_{\bar{1}} - r -1 , \ldots, \ell_{\bar{n}} - r - n)\\
        = &\mathrm{ev}_{\widetilde{\lambda} + \rho_{m+1|n}} p
    \end{align*}
    (Note that here we set $\mathrm{ev}_{a, \lambda+\rho_{m|n}}(p) \coloneqq p(a, \lambda_1 + m, ..., \lambda_m + 1, \lambda_{\bar{1}} - 1, ..., \lambda_{\bar{1}} - n)$ for all $p$).
    As a result, setting $p = \chi_{m+1|n}(z)$, and noting that $\chi_{a, \lambda}(\mathcal{E}^i \otimes z') = \mathrm{ev}_{a, \lambda + \rho_{m|n}}(\chi_{0, m|n}(\mathcal{E}^i \otimes z'))$, we have
    \[ \chi_{a,\lambda}(\xi(z))=\chi_{a,\lambda} \chi_{0, m|n}^{-1}\tau \chi_{m+1|n}(z) = \mathrm{ev}_{\widetilde{\lambda} + \rho_{m+1|n}}(\chi_{m+1|n}(z))= \chi_{\widetilde{\lambda}}(z).\]
    This completes the proof.
\end{proof}

\section{Super version of Newton's formula and Capelli-type Berezinians}\label{sec: capelli berezinians}

We have now computed the image of the center of $U(\mathfrak{gl}(m+1|n))$ under $\varphi$, and now we will describe the images of individual generators, as in \cite{GR_2021}. In \cite{GR_2021}, the images of the Gelfand generators are computed explicitly using a progression of identities related to Capelli generators, Capelli determinants, and an analogue of Newton's formula for $\mathfrak{gl}(m)$ (see \cite{Umeda_1998}). In this section, we extend this approach to the super setting and explicitly compute the images of the Gelfand generators under $\varphi$, though some care is required given that the center of $U(\mathfrak{gl}(m+1|n))$ is infinitely generated. Along the way we appeal to the theory of Yangians to reprove a super-version of Newton's formula that will be useful for us. 

\subsection{Yangians for \texorpdfstring{$\mathfrak{gl}(m|n)$}{gl(m|n)}}
Let us recall the Yangian $Y(\mathfrak{gl}(m|n))$ as defined in \cite{Nazarov_1991} (also see \cite{Gow_2005}). This is the $\mathbb{Z}_2$-graded associative algebra over $\mathbb{C}$ with generators
\begin{align*}
    \{T_{ij}^{(r)}: i,j\in I, r\ge 1\}
\end{align*}
and defining relations
\begin{align*}
    [T_{ij}^{(r)}, T_{kl}^{(s)}] = (-1)^{|i||j|+|i||k|+|j||k|}\sum_{p=0}^{\min(r,s)- 1} (T_{kj}^{(p)}T_{il}^{(r+s-1-p)}-T_{kj}^{(r+s-1-p)}T_{il}^{(p)}).
\end{align*}
The $\mathbb{Z}_2$-grading is given by $|T_{ij}^{(r)}| = |i|+|j|$.
We define the formal power series
\[T_{ij}(u) \coloneqq \delta_{ij}+ T_{ij}^{(1)}u^{-1} + T_{ij}^{(2)}u^{-2}+\cdots\]
and the matrix $T(u) \coloneqq [T_{ij}(u)]\in \mathrm{Mat}_{m|n}(Y(\mathfrak{gl}(m|n)))$. We can also identify $T(u)$ with
\begin{align*}
    \sum_{i,j\in I}T_{ij}(u)\otimes e_{ij}(-1)^{|j|(|i|+1)}.
\end{align*}

Note that some authors drop the sign term in the definition of $T(u)$ (e.g., Equation 2.12 in \cite{nazarov2020yangian}). As a result, these authors use a different sign convention for the comultiplication $\Delta$ defined below. However, this does not matter for our purposes because we never need to use the Hopf algebra structure of the Yangian.

In general, we can identify an operator $\sum A_{ij}\otimes e_{ij}(-1)^{|j|(|i|+1)}$ in $Y(\mathfrak{gl}(m|n))[\![ u^{-1}]\!] \otimes \mathrm{End} \mathbb{C}^{m|n}$ with the matrix $[A_{ij}]$. The extra signs are inserted to let the product of two matrices be calculated in the usual way. The Yangian is a Hopf algebra with comultiplication
\begin{align*}
    \Delta: T_{ij}(u)\mapsto \sum_{k\in I} T_{ik}(u)\otimes T_{kj}(u),
\end{align*}
antipode $S: T(u)\mapsto T^{-1}(u)$, and counit $T(u)\mapsto 1$. 

We define the power series with coefficients in the Yangian $Y(\mathfrak{gl}(m|n))$ of $\mathfrak{gl}(m|n)$ given by
\begin{align*}
    Z_h(u) \coloneqq 1+ \mathrm{str}\left(\frac{T(u+h) - T(u)}{h} \cdot T^{-1}(u) \right)
\end{align*}
for any $h\in \mathbb{C}$, where $\mathrm{str}(A)$ denotes the supertrace of a matrix $A$. The $h$ here has no relation to the $h$ that appears in \cite{Nazarov_1991}. Define
\begin{align}\label{eqn: Z(u)}
    Z(u) \coloneqq \lim_{h \to m-n}Z_h(u).    
\end{align}
For $m \ne n$, $Z(u) = Z_{m-n}(u)$ while for $m = n$ we have
\begin{align*}
    Z(u) = 1 + \mathrm{str}\left(\left(\frac{\mathrm{d}}{\mathrm{d}u}T(u)\right)T^{-1}(u)\right).
\end{align*}
Note that the same series $Z(u)$ was defined in \cite{Nazarov_1991}; however it was not written in terms of this limit construction. The coefficients of the series $Z(u)$ are known to generate the center of $Y(\mathfrak{gl}(m|n))$ (see \cite{Nazarov_1991} and \cite{Gow_2007}).

Let $B(u)$ be the quantum Berezinian of $T(u)$, defined by
\begin{align*}
    B(u) \coloneqq \mathrm{Det}[T_{ij}(u+(m-n-j))]_{i,j=1,\ldots m}\times \mathrm{Det}[T_{ij}^*(u-(m+n-j+1))]_{i,j=m+1,\ldots, m+n},
\end{align*}
where $[T_{ij}^*(u)]= T^*(u)= (T^{-1}(u))^{st}$ and $A^{st}= [A_{ji}(-1)^{|i|(|j|+1)}]$ denotes matrix supertransposition. By the determinant of a matrix $X = [X_{ab}]_{a,b=1,\ldots, N}$ with noncommuting entries, we mean the sum
\[\mathrm{Det} X \coloneqq \sum_{\sigma \in S_N}\mathrm{sgn}(\sigma) X_{\sigma(1),1}\cdots X_{\sigma(N),N}.\]
It is shown in \cite{Nazarov_1991} that
\begin{align}\label{eqn: quantum liouville}
    Z(u)=\frac{B(u+1)}{B(u)},
\end{align}
which could be referred to as a ``quantum Liouville formula''. Additionally, there is a projection homomorphism $\pi_{m|n}: Y(\mathfrak{gl}(m|n))\to U(\mathfrak{gl}(m|n))$ given by 
\begin{align}\label{eqn: projection}
    T_{ij}(u)\mapsto \delta_{ij} + e_{ij}(-1)^{|i|}u^{-1}.
\end{align}

\subsection{Super Newton's formula}
We can use these tools from Yangians to generalize the Newton's formula for $\mathfrak{gl}(m)$ (see \cite{Umeda_1998} and Theorem 7.1.3 in \cite{Molev_2007}) to $\mathfrak{gl}(m|n)$. Let 
\[\mathbf{E} \coloneqq ((-1)^{|i|}e_{ij})_{i,j\in I}\]
and define
\begin{align*}
    C_{m|n}(T) \coloneqq& \mathrm{Det}[\mathbf{E}_{ij}-T-i+1]_{i,j = 1,\ldots,m}\times \mathrm{Det}[(\mathbf{E}_{\bar{i}\bar{j}}-T-m+i)^{*}]_{i,j=1,\ldots, n} \\ =& \sum_{\sigma\in S_m}\mathrm{sgn}(\sigma) \left(\mathbf{E} - T\right)_{\sigma(1),1}\cdots (\mathbf{E} -T-m+1)_{\sigma(m),m}\\
    &\times \sum_{\tau\in S_n}\mathrm{sgn}(\tau)((\mathbf{E} - T -m+1)^{-1})^{st}_{m+\tau(1),m+1}\cdots ((\mathbf{E} - T-(m-n))^{-1})^{st}_{m+\tau(n),n}
\end{align*}
to be the {\it Capelli Berezinian} of $\mathfrak{gl}(m|n)$.
It follows from the definitions that $C_{m|n}(T) = Q(T) \pi_{m|n}(B(-T-(m-n)+1))$ for the rational function $Q(T)$ given by
\begin{align}\label{eqn: Q(T)}
        Q(T) = \begin{cases}
        T(T+1) \cdots (T+(m-n-1)) &\text{for $m> n$}\\
        1 &\text{for $m=n$}\\
        (T-(n-m))^{-1}(T-(n-m-1))^{-1}\cdots (T-1)^{-1} &\text{for $n>m$}.
        \end{cases}
\end{align}
It is well known that the coefficients of $\pi_{m|n}(B(T))$ generate $Z(U(\mathfrak{gl}(m|n)))$. Hence the coefficients of $C_{m|n}(T)$ also generate $Z(U(\mathfrak{gl}(m|n)))$. We call these coefficients the {\it Capelli generators} of $Z(U(\mathfrak{gl}(m|n)))$. Since the coefficients of $C_{m|n}(T)$ are central, we can consider $C_{m|n}$ as a function from the algebra of Laurent series in $T$ with coefficients in $U(\mathfrak{gl}(m|n))$ to itself.
The Newton's formulas proved in \cite{Umeda_1998} and Theorem 7.1.3 of \cite{Molev_2007} relate the Capelli Berezinian (called the Capelli determinant when $n=0$) to the Gelfand invariants $G_k^{\mathfrak{gl}(m)}$. We reprove analogs of these results to $\mathfrak{gl}(m|n)$ using the same method as \cite{Molev_2007}.

The Gelfand invariants of $\mathfrak{gl}(m|n)$ are given by
\[G_k^{\mathfrak{gl}(m|n)} \coloneqq \sum_{i_1,\ldots,i_k \in \hat{I}} (-1)^{|i_2|+\cdots + |i_k|}e_{i_1i_2}\otimes e_{i_2i_3}\otimes \cdots e_{i_ki_1}\]
for $k\ge 1$, where the sum ranges over all $k$-tuples $(i_1,\ldots, i_k)$ with terms in $I$. It is known that the Gelfand invariants $G_k$ for $k=0,1, \ldots$ generate $Z(U(\mathfrak{gl}(m|n)))$ (see for example \cite{Rao_2021}). We can also write $G_k^{\mathfrak{gl}(m|n)} = \mathrm{str} \mathbf{E}^k$. Indeed, we have
\begin{align*}
    \mathrm{str}\mathbf{E}^k &=\sum_{i\in I} (-1)^{|i|}(\mathbf{E}^k)_{ii}\\
    &= \sum_{i\in I} (-1)^{|i|}\sum_{i_2,\ldots, i_k\in I} \mathbf{E}_{ii_2}\mathbf{E}_{i_2i_3}\cdots \mathbf{E}_{i_ki}\\
    &=\sum_{i_1,\ldots, i_k\in I} (-1)^{|i_1| + (|i_1|+\cdots + |i_{k-1}| + |i_k|)}e_{i_1i_2}e_{i_2i_3}\cdots e_{i_ki_1}\\
    &=\sum_{i_1,\ldots, i_k\in I} (-1)^{|i_2|+\cdots + |i_k| }e_{i_1i_2}e_{i_2i_3}\cdots e_{i_ki_1}\\
    &= G_k^{\mathfrak{gl}(m|n)}. 
\end{align*}
\begin{thmsub}[Newton's Formulas for $\mathfrak{gl}(m|n)$]\label{thm : super newton}
    We have
    \begin{align*}
        \frac{C_{m|n}(T-(m-n))}{C_{m|n}(T-(m-n)+1)} = 1- \sum_{k=0}^\infty G_k^{\mathfrak{gl}(m|n)}T^{-1-k}.
    \end{align*}
\end{thmsub}
\begin{proof}
   
    Since $C_{m|n}(T) = Q(T) \pi_{m|n}(B(-T-(m-n)+1))$, (\ref{eqn: Q(T)}) and that $G_0^{\mathfrak{gl}(m|n)}=m-n$ yields that the identity is equivalent to 
    \begin{align*}
        \frac{\pi_{m|n}(B(-T+1))}{\pi_{m|n}(B(-T))} = 1 + \frac{1}{-T+(m-n)}\sum_{k=1}^\infty \mathrm{str} \mathbf{E}^k T^{-k}.
    \end{align*}
    Using (\ref{eqn: Z(u)}) and (\ref{eqn: projection}), we get that
\begin{align*}
    \pi_{m|n}(Z(-u)) &= \lim_{h \to m-n}1+\mathrm{str}\left(\frac{\mathbf{E} u^{-1} - \mathbf{E}(u-h)^{-1}}{h}(1-\mathbf{E} u^{-1})^{-1}\right)\\
    &= \lim_{h \to {m-n}}1 - \frac{1}{u(u-h)}\sum_{k=0}^\infty \mathrm{str} \mathbf{E}^{k+1} u^{-k} \\
    &=1 - \frac{1}{u-(m-n)}\sum_{k=1}^\infty \mathrm{str} \mathbf{E}^k u^{-k}.
\end{align*}
Then the identity follows from (\ref{eqn: quantum liouville}).
\end{proof}
\subsection{Images of Central Elements}
Next, we find the image of the Capelli Berezinian under the Harish-Chandra homomorphism, which combined with Theorem~\ref{thm : HC homomorphism}, will let us find the image of the Capelli Berezinian under $\varphi$.
\begin{propsub}\label{image of capelli under HC}
    We have
    \begin{align*}
        \chi_{m|n}(C_{m|n}(-T)) = (T+\ell_1-m)\cdots (T+\ell_m-m) (T-\ell_{\bar{1}}-m)^{-1}\cdots (T-\ell_{\bar{n}}-m)^{-1}.
    \end{align*}
\end{propsub}
\begin{proof}
    Note $C_{m|n}(T)$ is the product of two parts, one part summing over $\sigma\in S_m$, and one summing over $\tau\in S_n$. However, note that in both parts, only the summand corresponding to the identity permutation has a nonzero image under $\chi_{m|n}$, because in both $\mathbf{E}+T$ and
    \begin{align*}
        (\mathbf{E} + T)^{-1} = T^{-1}-\mathbf{E} T^{-2}+ \mathbf{E}^2 T^{-3}-\cdots,
    \end{align*}
    the only elements that stabilize the weight spaces of a Verma module $M(\lambda)$ are along the diagonal. To see this, define $r_{k}^{\mathfrak{gl}(m|n)}(a,b) = \mathbf{E}^k_{ab}$. Then
    \[r_{k}^{\mathfrak{gl}(m|n)}(a,b) = \sum_{i_1,\ldots, i_{k-1}\in I} (-1)^{|i_1|+\cdots + |i_{k-1}|}e_{ai_1}\otimes e_{i_1i_{2}}\otimes \cdots \otimes e_{i_{k-1}b}.\]
    and it follows that acting by $r_k^{\mathfrak{gl}(m|n)}(a,b)$ raises a weight by $\delta_a - \delta_b$, hence stabilizes the weight spaces if and only if $a=b$. Then, the product of the two terms gives the result. 
\end{proof}
Define $C_{\varphi}(T)$ via the identity $C_{\varphi}(\varphi(T)) = \varphi(C_{m+1|n}(T))$. Since the coefficients of the Capelli determinant are central and $\varphi$ maps $Z(U(\mathfrak{gl}(m+1|n)))$ to $Z(\mathcal{D}'(m|n)\otimes U(\mathfrak{gl}(m|n)))$, the domain of $C_{\varphi}(T)$ consists of Laurent series in $\varphi(T)$ with coefficients in $\mathcal{D}'(m|n)\otimes U(\mathfrak{gl}(m|n))$.
For convenience we will write $T$ for $\varphi(T)$.
\begin{thmsub}\label{capelli under varphi}
    We have
    \begin{align*}
        C_{\varphi}(T+R) = (\mathcal{E} - T)C_{m|n}(T+1).
    \end{align*}
    Consequently, $C_{\varphi}(\mathcal{E} + R)=0$.
\end{thmsub}
\begin{proof}
    The identity is equivalent to
    \begin{align*}
        \varphi(C_{m+1}(-T)) = (\mathcal{E} +T +R)C_{m|n}(-T-R+1).
    \end{align*}
    To prove this, we use Proposition~\ref{image of capelli under HC}, the corresponding version for $m+1|n$, Theorem \ref{thm : HC homomorphism}, and that $\chi_{0,n}(R)=\ell$.
\end{proof}
Define
\[R_2 = (\mathcal{E} + m-n)\otimes 1, \]
and
\begin{align*}
    A_k(s) &= \sum_{g=0}^{k-s}\binom{k}{g} R^g R_2^{k-s-g}.
\end{align*}
We can combine the previous theorem with the Newton's formula to find the images of the Gelfand invariants under $\varphi$.
\begin{thmsub}\label{thm : gelfand}
    The following formula holds for all positive integers $k$:
    \begin{align*}
        \varphi(G_k^{\mathfrak{gl}(m+1|n)}) &= A_k(1)(\mathcal{E}\otimes 1) + (m+1-n)R^{k} + \sum_{g=0}^{k-1}\binom{k}{g}R^g\left(1\otimes G_{k-g}^{\mathfrak{gl}(m|n)}\right)\\
        &\phantom{=}-\sum_{s=2}^k \left(1\otimes G_{s-1}^{\mathfrak{gl}(m|n)}\right)A_k(s).
    \end{align*}
\end{thmsub}
\begin{proof}
    Theorem~\ref{capelli under varphi} implies that
    \begin{align}\label{ratio of capelli}
    \frac{C_{\varphi}(T-(m-n)-1)}{C_{\varphi}(T-(m-n))} = \left(1-\frac{1}{T-R-R_2}\right)\frac{C_{m|n}(T-R-(m-n))}{C_{m|n}(T-R-(m-n)+1)}.
    \end{align}
    Applying $\varphi$ to the Newton's formula for $\mathfrak{gl}(m+1|n)$ gives
    \begin{align*}
    \frac{C_{\varphi}(T-(m-n)-1)}{C_{\varphi}(T-(m-n))} &= 1-\sum_{k=0}^\infty \varphi(G_{k}^{\mathfrak{gl}(m+1|n)})T^{-1-k} \\ &= 1- (m-n+1)T^{-1} - T^{-1}\sum_{k=1}^\infty \varphi(G_{k}^{\mathfrak{gl}(m+1|n)})T^{-k}.
\end{align*}
Analogously we can express the right hand side of (\ref{capelli under varphi}) as a power series in $T^{-1}$. We complete the proof by comparing and computing the coefficients of $T^{-k}$ on both sides.
\end{proof}

\section{Kernel of \texorpdfstring{$\varphi_{R_1}$}{varphiR1}}\label{sec: kernel}
In this section, we find the kernel of the map $\varphi_{R_1}$. For this section, we write $\varphi=\varphi_{R_1}$ and $G_1 = G_1^{\mathfrak{gl}(m+1|n)}$ for convenience.
\begin{thm}\label{thm : kernel}
    Let $(G_1)$ be the two sided ideal in $U(\mathfrak{gl}(m+1|n))$ generated by $G_1$. Then the kernel of $\varphi$ is $(G_1)$.
\end{thm}
\begin{proof}
    Recall from Proposition~\ref{prop: commutative diagram} that $\varphi = \pi^g \varphi^s$. The kernel of $\pi^g$ is clearly $(G_1)$, so we just need to show that $\varphi^s$ is injective. We construct a left inverse for $\varphi^s$, considered as a \textit{Lie superalgebra} homomorphism. In particular, this means we view $U(\mathfrak{gl}(m+1|n))$ and $U(\mathfrak{gl}(m|n))$ as Lie algebras with bracket given by $[x,y] = xy-(-1)^{|x||y|}yx$.
    
    Let $\mathcal{D}''(m|n)$ be the subalgebra of $\mathcal{D}'(m|n)$ generated by $\{t_a\partial_b : a,b\in \hat{I}, a\ne b\}\cup\{t_0\partial_0 - (-1)^{|a|}t_a\partial_a: a\in \hat{I}\}$. 
    Consider the homomorphism $1\otimes \varepsilon : \mathcal{D}''(m|n) \otimes U(\mathfrak{gl}(m|n))\to \mathcal{D}''(m|n)$, where $\varepsilon$ is the counit map of $U(\mathfrak{gl}(m|n))$. In particular $(1\otimes \varepsilon)(p \otimes x) = 0$ for $p\in \mathcal{D}''(m|n)$ and $x\in \mathfrak{gl}(m|n)$ and $(1\otimes \varepsilon)(p\otimes 1) = p$. 
    Consider the homomorphism $(1\otimes \varepsilon) \circ \varphi^s: U(\mathfrak{sl}(m+1|n)) \to \mathcal{D}''(m|n)$. Corollary~\ref{cor: sl} gives us that this homomorphism is generated by
    \begin{align*}
        e_{ab} &\mapsto t_a\partial_b \tag{for $a\ne b$}\\
        e_{00}-(-1)^{|a|}e_{aa} &\mapsto t_{0}\partial_0- (-1)^{|a|}t_a\partial_a.
    \end{align*}
    We claim that this homomorphism has a left inverse $\psi$. Letting $\mathcal{D}'''(m|n)$ be the algebra generated by $\{t_a\partial_b : a,b\in \hat{I}\}$, we first define $\psi': \mathcal{D}'''(m|n)\to U(\mathfrak{gl}(m+1|n))$ and then its restriction $\psi: \mathcal{D}''(m|n)\to U(\mathfrak{sl}(m+1|n))$ will be a left inverse to $(1\otimes \varepsilon)\circ \varphi^s$. The map $\psi'$ is generated by
    \begin{align*}
        t_a\partial_b &\mapsto e_{ab}
    \end{align*}
    for $a,b\in \hat{I}$. To check $\psi$ is a homomorphism, we check $\psi'([t_a\partial_b, t_c\partial_d]) = [e_{ab},e_{cd}]$ for all $a,b,c,d\in \hat{I}$.

    If $a\ne d, b\ne c$, then we have
    \begin{align*}
        [t_a\partial_b, t_c\partial_d] &= t_a\partial_bt_c\partial_d - (-1)^{(|a|+|b|)(|c|+|d|)} t_c\partial_d t_a\partial_b\\
        &=(-1)^{|b||c|}t_a t_c \partial_b\partial_d - (-1)^{(|a||c| + |b||c| + |a||c|)}t_ct_a \partial_d \partial_b\\
        &= (-1)^{|b||c|}t_at_c \partial_b t_d- (-1)^{|b||c|}t_at_c \partial_b t_d = 0.
    \end{align*}
    Now suppose $a\ne d$, $b=c$. Note $[e_{ab},e_{bd}] = e_{ad}$. Thus, we show $[t_{a}\partial_b, t_b\partial_d] = t_a\partial_d$. Indeed, we have
    \begin{align*}
        [t_a\partial_b, t_b\partial_d] &= t_a\partial_b t_b \partial_d - (-1)^{(|a|+|b|)(|b|+|d|)}t_b\partial_d t_a\partial_b\\
        &= t_a\partial_b t_b \partial_d -(-1)^{|b|} t_a t_b\partial_b \partial_d\\
        &= t_a\partial_d.
    \end{align*}
    If $a=d, b\ne c$, reversing the order gives the previous case.
    If $a=d, b=c$, then $[e_{ab}, e_{ba}] = e_{aa}- (-1)^{|a|+|b|}e_{bb}$. We also have
    \begin{align*}
        [t_a\partial_b, t_b\partial_a] &= t_a\partial_b t_b\partial_a - (-1)^{|a|+|b|}t_b\partial_a t_a\partial_b \\
        &= t_a (1+(-1)^{|b|}t_b\partial_b)\partial_a - (-1)^{|a|+|b|} t_b\partial_a t_a\partial_b\\
        &= t_a\partial_a - (-1)^{|a| + |b|} t_b\partial_b (\partial_at_a - (-1)^{|a|}t_a\partial_a)\\
        &= t_a\partial_a - (-1)^{|a|+|b|}t_b\partial_b.
    \end{align*}
    Thus $\psi'$, hence $\psi$ is a homomorphism. It is easily checked that $\psi \circ (1\otimes \varepsilon) \circ \varphi^s=\mathrm{Id}$, so $\varphi^s$ is injective. This completes the proof.
\end{proof}

\bibliographystyle{amsalpha}
\bibliography{references}   
\vfill
\newpage
\appendix
\section{Proof of Theorem \texorpdfstring{\ref{thm: homomorphism}}{3}}\label{proof}
In this appendix we provide the proof of Theorem \ref{thm: homomorphism}. 
\begin{proof}[Proof of Theorem \ref{thm: homomorphism}]
    Clearly $\varphi$ maps even elements to even elements and odd elements to odd elements. Thus, we only need to check that for basis elements $x$ and $y$, we have
    \begin{align}
    \label{eqn: commutators}
        [\varphi(x),\varphi(y)] = \varphi([x,y]).
    \end{align}
    We have multiple cases to check. Here $a,b,c,d$ will always denote elements of $I$. Before we start verifying the cases, note that since $R$ is central, all brackets with it are zero.
    \begin{description}
        \item[Case 1] $[e_{ab},e_{ba}]$. \vskip 5pt 
        \noindent If $a = b$ then both brackets in (\ref{eqn: commutators}) are obviously zero. 
        
        \noindent If $a\ne b$ then $[e_{ab},e_{ba}] = e_{aa} - (-1)^{|a|+|b|}e_{bb}$. Thus, 
        \begin{align*}
        \varphi([e_{ab},e_{ba}]) &= \Big(t_a\partial_a-(-1)^{|a|+|b|}t_b\partial_b\Big)\otimes 1 + 1\otimes \Big(e_{aa}-(-1)^{|a|+|b|}e_{bb}\Big) 
        \end{align*}
    Now we show
    \[[t_a\partial_b, t_b\partial_a] = t_a\partial_a - (-1)^{|a|+|b|}t_b\partial_b\]
    for $a\ne b$.
    The left hand side is
    \begin{align}\label{eqn: commutator of partials I}
        t_a\partial_bt_b\partial_a - (-1)^{|a||b|}t_b\partial_at_a\partial_b.
    \end{align}
    Now $\partial_bt_b = 1 + (-1)^{|b|}t_b\partial_b$, so $t_a\partial_bt_b\partial_a = t_a\partial_a + (-1)^{|b|}t_at_b\partial_b\partial_a$. Using supercommutativity, this becomes $t_a\partial_a + (-1)^{|b|}t_a\partial_a t_b\partial_b$. A similar process with the second term gives us $t_b\partial_at_a\partial_b =t_b\partial_b +(-1)^{|a|}t_a\partial_at_b\partial_b$. Plugging these back into (\ref{eqn: commutator of partials I}), we get
    \begin{align*}
        [t_a\partial_b, t_b\partial_a] &= t_a\partial_a+ (-1)^{|b|}t_a\partial_a t_b\partial_b - (-1)^{|a|+|b|}t_b\partial_b - (-1)^{|b|}t_a\partial_a t_b\partial_b\\
        &=t_a\partial_a - (-1)^{|a|+|b|}t_b\partial_b.
    \end{align*}
    Thus,
    \begin{align*}
        [\varphi(e_{ab}),\varphi(e_{ba})] &= [t_a\partial_b\otimes 1 + 1\otimes e_{ab},t_b\partial_a\otimes 1 + 1\otimes e_{ba}]\\
        &= \Big(t_a\partial_a-(-1)^{|a|+|b|}t_b\partial_b\Big)\otimes 1 + 1\otimes \Big(e_{aa}-(-1)^{|a|+|b|}e_{bb}\Big)\\
        &= \varphi([e_{ab},e_{ba}]),
    \end{align*}
    verifying (\ref{eqn: commutators}).

        \item[Case 2] $[e_{ab},e_{bc}]$, $a\ne c$\vskip 5pt
        \noindent 
        First, we show
        \[[t_a\partial_b, t_b\partial_c] = t_a\partial_c.\]
        The left hand side is 
        \begin{align}\label{eqn: commutator of partials II}
            t_a\partial_bt_b\partial_c - (-1)^{(|a|+|b|)(|b| + |c|)}t_b\partial_ct_a\partial_b.
        \end{align}
        We have
        \begin{align*}
            t_a\partial_bt_b\partial_c &= (-1)^{|b||c|}t_a\partial_b \partial_c t_b\\
            &= t_a\partial_c \partial_bt_b.
        \end{align*}
        Also,
        \begin{align*}
            t_b\partial_ct_a\partial_b &= (-1)^{|b||c|}\partial_c t_bt_a\partial_b\\
            &= (-1)^{|b||c| + |b||a|}\partial_ct_at_b \partial_b\\
            &= (-1)^{|b||c| + |b||a| + |a||c|}t_a\partial_c t_b\partial_b.
        \end{align*}
        Thus $(-1)^{(|a|+|b|)(|b| + |c|)}t_b\partial_c t_a\partial_b = (-1)^{|b|^2}t_a\partial_c t_b\partial_b = (-1)^{|b|}t_a\partial_c t_b\partial_c$. Plugging in to (\ref{eqn: commutator of partials II}), we get
        \begin{align*}
            [t_a\partial_b, t_b\partial_c] = t_a\partial_c(\partial_bt_b - (-1)^{|b|}t_b\partial_b) = t_a\partial_c.
        \end{align*}
        Then, we have
        \begin{align*}
            [\varphi(e_{ab}), \varphi(e_{bc})] &= [t_a\partial_b \otimes 1 + 1 \otimes e_{ab}, t_b\partial_c \otimes 1 + 1 \otimes e_{bc}] \\
            &= t_a\partial_c + 1 \otimes e_{ac}  \\
            &= \varphi(e_{ac})  = \varphi([e_{ab}, e_{bc}]),
        \end{align*}
        as desired.

        \item[Case 3] $[e_{ab},e_{ca}]$, $c\ne b$\vskip 5pt
        \noindent Swapping the order of the two terms gives us Case 2. Note that all signs picked up in the process get canceled.

        \item[Case 4] $[e_{ab},e_{cd}]$, $c\ne b$ and $a\ne d$\vskip 5pt
        \noindent 
        First, we show
        \begin{align*}
            [t_{a}\partial_b, t_c\partial_d] = 0.
        \end{align*}
        The left hand side is
        \begin{align}\label{eqn: commutator of partials III}
            t_a\partial_bt_c\partial_d - (-1)^{(|a|+|b|)(|c|+|d|)}t_c\partial_dt_a\partial_b.
        \end{align}
        We have
        \begin{align*}
            t_c\partial_d t_a\partial_b &= (-1)^{|a|(|c|+|d|)}t_at_c\partial_d \partial_b\\
            &= (-1)^{|a|(|c|+|d|)+|b|(|c|+|d|)}t_a\partial_bt_c\partial_d\\
            &= (-1)^{(|a|+|b|)(|c|+|d|)}t_a\partial_bt_c\partial_d.
        \end{align*}
        Plugging back into (\ref{eqn: commutator of partials III}), we get the desired result. Thus, it follows that $\varphi(e_{ab})$ commutes with $\varphi(e_{cd})$. Hence, $[\varphi(e_{ab}), \varphi(e_{cd})] = 0 = \varphi([e_{ab}, e_{cd}])$ as needed.

        \item[Case 5] $[e_{ab}, e_{b0}]$\vskip 5pt
        \noindent
        Since $R$ bracketed with anything is zero, we have
        \begin{align*}
        [\varphi(e_{ab}),\varphi(e_{b0})] &= \left[t_{a}\partial_b\otimes 1 + 1\otimes e_{ab},t_b\partial_{0}\otimes 1 -\sum_{j} (-1)^{|b||j|}\frac{t_j}{t_0}\otimes e_{bj}\right]\\
            &=[t_{a}\partial_b\otimes 1,t_b\partial_{0}\otimes 1] - (-1)^{|b|}\left[t_{a}\partial_b\otimes 1 , \frac{t_b}{t_0}\otimes e_{bb}\right] \\ &\phantom{=} - \sum_{j \ne a} (-1)^{|b||j|}\left[1\otimes e_{ab}, \frac{t_{j}}{t_0}\otimes e_{bj}\right] \\ 
            &\phantom{=}-(-1)^{|a||b|}\left[1\otimes e_{ab},\frac{t_{a}}{t_0}\otimes e_{ba}\right]\\
            &= t_{a}\partial_0 \otimes 1 - (-1)^{|b|}\frac{t_a}{t_0} \otimes e_{bb} +(-1)^{|b|}\frac{t_a}{t_0} \otimes e_{bb} \\ 
            &\phantom{=} - \sum_{j} (-1)^{|a||j|} \frac{t_j}{t_0} \otimes e_{aj} \\
            &= t_{a}\partial_0 \otimes 1 - \sum_{j} (-1)^{|a||j|} \frac{t_j}{t_0} \otimes e_{aj}\\
            &= \varphi([e_{ab},e_{b0}])
        \end{align*} as desired. We now prove each of the four brackets used here. That $[t_a\partial_b\otimes 1, t_b\partial_0 \otimes 1] = t_a\partial_0\otimes 1$ was proved in Case 2. Next, we have
        \begin{align*}
            \left[t_{a}\partial_b\otimes 1 , \frac{t_b}{t_0}\otimes e_{bb}\right] &= \left[t_a\partial_b, \frac{t_b}{t_0}\right]\otimes e_{bb}
        \end{align*}
        since $e_{bb}$ is even. We have
        \begin{align*}
            \left[t_a\partial_b, \frac{t_b}{t_0}\right] &= t_a\partial_b \frac{t_b}{t_0} - (-1)^{|b|(|a|+|b|)}\frac{t_b}{t_0}t_a\partial_b\\
            &= \frac{t_a}{t_0}\partial_bt_b - (-1)^{|b|(|a|+|b|)+ |a||b|}\frac{t_a}{t_0}t_b\partial_b\\
            &= \frac{t_a}{t_0}(\partial_bt_b - (-1)^{|b|}t_b\partial_b) = \frac{t_a}{t_0}.
        \end{align*}
        This verifies the second bracket. The third bracket is
        \[\left[1\otimes e_{ab}, \frac{t_j}{t_0}\otimes e_{bj}\right]\]
        for $j\ne a$. Note that $e_{bj}e_{ab} = 0$. Thus the bracket is
        \[(-1)^{(|a|+|b|)|j|} \frac{t_j}{t_0}\otimes e_{aj}.\]
        Multiplying by $(-1)^{|b||j|}$ gives
        \[(-1)^{|a||j|}\frac{t_j}{t_0}\otimes e_{aj},\]
        as desired. The final bracket is
        \begin{align*}
        \left[1\otimes e_{ab},\frac{t_{a}}{t_0}\otimes e_{ba}\right] &= (-1)^{|a|(|a|+|b|)}\frac{t_a}{t_0}\otimes e_{aa} - (-1)^{(|a|+|b|)(|b|)}\frac{t_a}{t_0}\otimes e_{bb}.
        \end{align*}
        Multiplying by $(-1)^{|a||b|}$, we get
        \[(-1)^{|a||a|}\frac{t_a}{t_0}\otimes e_aa - (-1)^{|b|}\frac{t_a}{t_0}\otimes e_{bb},\]
        as desired.

        \item[Case 6] $[e_{ab}, e_{0a}]$ \vskip 5pt
        \noindent
        We show
        \begin{align}\label{commutator of partials result}
            [t_a\partial_b, t_0\partial_a] = -(-1)^{(|a|+|b|)|a|}t_0\partial_b.    
        \end{align}
        The left hand side is
        \begin{align}\label{eqn: commutator of partials IV}
            t_a\partial_bt_0\partial_a - (-1)^{(|a|+|b|)|a|}t_0\partial_at_a\partial_b.
        \end{align}
        We have $t_a\partial_bt_0\partial_a = (-1)^{|a||b|}t_0\partial_bt_a\partial_a = (-1)^{|a|}(-1)^{(|a|+|b|)|a|}t_0\partial_b t_a\partial_a$. Also $t_0\partial_at_a\partial_b = t_0\partial_b\partial_at_a$. Plugging these back into (\ref{eqn: commutator of partials IV}), we get (\ref{commutator of partials result}). Then, we have
        \begin{align*}
            [\varphi(e_{ab}), \varphi(e_{0a})] &= [t_a\partial_b \otimes 1 + 1 \otimes e_{ab}, t_0\partial_a \otimes 1] \\
            &= -t_0\partial_b \otimes 1 = \varphi(-e_{0b}) \\
            &= \varphi([e_{ab}, e_{0a}]),
        \end{align*}as desired.
        
        \item[Case 7] $[e_{ab}, e_{00}]$\vskip 5pt
        \noindent We have $[e_{ab}, e_{00}]=0$, so $\varphi([e_{ab}, e_{00}])=0$. Note that $t_0\partial_0$ commutes with $t_a\partial_b$, so $[t_a\partial_b, t_0\partial_0]=0$. Since $R$ is central, we have
        \begin{align*}
            [\varphi(e_{ab}),\varphi(e_{00})] &= [t_a\partial_b\otimes 1 + 1\otimes e_{ab}+\delta_{ab}(-1)^{|a||b|}R, t_0\partial_0\otimes 1 +R]\\
            &=0 = \varphi([e_{ab},e_{00}]).
        \end{align*}
        \item[Case 8] $[e_{a0},e_{b0}]$\vskip 5pt
        \noindent
        We have 
        \begin{align*}
        [\varphi(e_{a0}), \varphi(e_{b0})] &= \left[t_a\partial_0 \otimes 1 - \sum_{j} (-1)^{|a||j|} \frac{t_j}{t_0} \otimes e_{aj}, t_b\partial_0 \otimes 1 - \sum_{j} (-1)^{|b||j|} \frac{t_j}{t_0} \otimes e_{bj}\right].
        \end{align*}
        We evaluate each sub-bracket one by one. It is easily checked that $[t_a\partial_0\otimes 1, t_b\partial_0\otimes 1]=0$. We have
        \begin{align*}
            \left[t_a\partial_0 \otimes 1, \frac{t_j}{t_0}\otimes e_{bj}\right] &= t_a\partial_0 \frac{t_j}{t_0}\otimes e_{bj} - (-1)^{|a||b|+ (|b| + |j|)|a|}\frac{t_j}{t_0}t_a\partial_0 \otimes e_{bj} \\
            &= t_at_j \left(\partial_0 \frac{1}{t_0}\right)\otimes e_{bj} - t_at_j\left(\frac{1}{t_0}\partial_0\right) \otimes e_{bj}.
        \end{align*}
        Similar to the identity $\partial_0t_0 - t_0\partial_0 = 1$, we have $\partial_0\frac{1}{t_0}-\frac{1}{t_0}\partial_0 = -\frac{1}{t_0^2}$, so the bracket evaluates to $-t_at_j$. Thus
        \[\left[t_a\partial_0\otimes 1, - \sum_{j}(-1)^{|b||j|}\frac{t_j}{t_0}\otimes e_{bj}\right] = \sum_{j}(-1)^{|b||j|}\frac{t_at_j}{t_0^2}\otimes e_{bj}.\]
        Next we evaluate 
        \[\left[\frac{t_i}{t_0}\otimes e_{ai}, \frac{t_j}{t_0}\otimes e_{bj}\right]\]
        for $i,j\in I$. If $i\ne b$ then $e_{ai}e_{bj}=0$. If $j\ne a$ then $e_{bj}e_{ai}=0$. Thus
        \begin{align*}
            \left[\frac{t_i}{t_0}\otimes e_{ai}, \frac{t_j}{t_0}\otimes e_{bj}\right] &= \delta_{ib}(-1)^{|j|(|a|+|b|)} \frac{t_bt_j}{t_0^2}\otimes e_{aj} - \delta_{aj}(-1)^{|a|||b| + (|b|+|a)|i|}\frac{t_at_i}{t_0^2}\otimes e_{bi}.
        \end{align*}
        Thus
        \begin{align*}
            \left[-\sum_{i}(-1)^{|a||i|} \frac{t_j}{t_0}\otimes e_{ai}, -\sum_{j}(-1)^{|b||j|}\frac{t_j}{t_0}\otimes e_{bj}\right]&= (-1)^{|a||b|}\sum_{j}(-1)^{|a||j|}\frac{t_bt_j}{t_0^2}\otimes e_{aj}\\
            &\phantom{=} -\sum_j (-1)^{|b||j|}\frac{t_at_j}{t_0^2}\otimes e_{bj}.
        \end{align*}
        Finally, 
        \begin{align*}
        \left[\frac{t_j}{t_0}\otimes e_{aj}, t_b\partial_0\otimes 1\right] &= (-1)^{(|a|+|j|)|b| + |j||b|} t_bt_j\left(\frac{1}{t_0}\partial_0\right) \otimes e_{aj} - (-1)^{|a||b|} t_bt_j\left(\partial_0 \frac{1}{t_0}\right)\otimes e_{aj}\\
        &= (-1)^{|a||b|} \frac{t_bt_j}{t_0^2} \otimes e_{aj}.
        \end{align*}
        Thus
        \[\left[ \sum_{j} (-1)^{|a||j|} \frac{t_j}{t_0} \otimes e_{aj}, t_b\partial_0 \otimes 1\right] =  - (-1)^{|a||b|}\sum_j (-1)^{|a||j|}\frac{t_bt_j}{t_0^2}\otimes e_{aj}.\]
        Adding everything back together, we see that everything cancels, so
        \[[\varphi(e_{a0}), \varphi(e_{b0})] = 0 = \varphi([e_{a0,e_{b0}}]).\]

        \item[Case 9] $[e_{a0},e_{0b}]$\vskip 5pt
        \noindent
       
        Now
        \begin{align*}
            [\varphi(e_{a0}), \varphi(e_{0b})] &= \left[t_a\partial_0 \otimes 1 - \sum_{j} (-1)^{|a||j|} \frac{t_j}{t_0} \otimes e_{aj}, t_0\partial_b \otimes 1\right]\\
            &= [t_a\partial_0, t_0\partial_b] \otimes 1 - \sum_{j}(-1)^{|a||j|}\left[\frac{t_j}{t_0} \otimes e_{aj} ,t_0\partial_b \otimes 1\right].
        \end{align*}
        But
        \begin{align*}
            \left[\frac{t_j}{t_0} \otimes e_{aj}, t_0\partial_b \otimes 1\right] &= \left(\frac{t_j}{t_0} \otimes e_{aj}\right)(t_0\partial_b \otimes 1) - (-1)^{|a||b|}(t_0\partial_b \otimes 1)\left(\frac{t_j}{t_0} \otimes e_{aj}\right) \\
            &= (-1)^{(|a|+|j|)|b|}\left(\frac{t_j}{t_0}t_0\partial_b\right) \otimes e_{aj} - (-1)^{|a||b|}(t_0\partial_b)(t_j/t_0) \otimes e_{aj} \\
            &= (-1)^{|a||b|}[(-1)^{|j||b|}(t_j\partial_b) -(\partial_bt_j)] \otimes e_{aj}
        \end{align*}
        If $j = b$ the expression in the brackets is
        \[-(\partial_b t_b - (-1)^{|b|}t_b\partial_b) = -1.\]
        If $j\ne b$ the expression in the brackets is
        \[(-1)^{|j||b|}(t_j\partial_b)-(\partial_bt_j) = (-1)^{|j||b|}(t_j\partial_b)-(-1)^{|j||b|}(t_j\partial_b) = 0.\]
        Plugging these back into the original expression, we get
        \begin{align*}
            [\varphi(e_{a0}), \varphi(e_{0b})] &= [t_a\partial_0, t_0\partial_b]\otimes 1 + 1\otimes e_{ab}.
        \end{align*}
        If $a\ne b$, from Case 2 we have $[t_a\partial_0,t_0\partial_b] = t_a\partial_b$, so
        \[ [\varphi(e_{a0}), \varphi(e_{0b})] = t_a\partial_b\otimes 1 + 1\otimes e_{ab} = \varphi([e_{a0},e_{0b}]). \]
        If $a = b$, we have
        \begin{align*}
            [t_a\partial_0, t_0\partial_a] &= t_a\partial_0t_0\partial_a - (-1)^{|a|}t_0\partial_a t_a\partial_0\\
            &=\partial_0t_0 t_a\partial_a - (-1)^{|a|} t_0\partial_0 \partial_a t_a\\
            &= \partial_0t_0 t_a\partial_a - (-1)^{|a|}t_0\partial_0 (1+(-1)^{|a|} t_a\partial_a)\\
            &= (-1)^{|a|}t_0\partial_0 + (\partial_0 t_0 - t_0\partial_0)t_a\partial_a\\
            &= t_a\partial_a - (-1)^{|a|}t_0\partial_0.
        \end{align*}
        Thus
        \begin{align*}
            [\varphi(e_{a0}), \varphi(e_{0a})] = t_a\partial_0 \otimes 1 - (-1)^{|a|}t_0\partial_0 \otimes 1 + 1\otimes e_{aa} = \varphi(e_{aa} - (-1)^{|a|}e_{00}) = \varphi([e_{a0},e_{0a}]).
        \end{align*}
        
        \item[Case 10] $[e_{a0},e_{00}]$\vskip 5pt
        \noindent 
        Since $R$ is central,
        \[ [\varphi(e_{a0}), \varphi(e_{00})] = \left[t_a\partial_{0}\otimes 1 -\sum_{j} (-1)^{|a||j|}\frac{t_j}{t_0}\otimes e_{aj}, t_0\partial_0 \otimes 1\right] \]
        Now,
        \[[t_a\partial_0, t_0\partial_0] = (t_a\partial_0)(t_0\partial_0) - (t_0\partial_0)(t_a\partial_0) = t_a(\partial_0t_0 - t_0\partial_0) \partial_0 = t_a\partial_0.\]
        Also,
        \[ \left[\frac{t_j}{t_0} \otimes e_{aj}, t_0\partial_0 \otimes 1\right] = t_j(\partial_0t_0- t_0 \partial_0)\frac{1}{t_0}\otimes e_{aj} = \frac{t_j}{t_0} \otimes e_{aj}.\]
        Thus
        \[[\varphi(e_{a0}), \varphi(e_{00})] = t_a\partial_0 \otimes 1 - \sum_j (-1)^{|a||j|}\frac{t_j}{t_0} \otimes e_{aj} = \varphi([e_{a0},e_{00}]).\]
        \item[Case 11] $[e_{0a},e_{0b}]$\vskip 5pt
        \noindent We have 
        \[\varphi([e_{0a},e_{0b}])=\varphi(0)=0.\]
        Also
        \begin{align*}
            [\varphi(e_{0a}),\varphi(e_{0b})] &= [t_0\partial_a\otimes 1, t_0\partial_b\otimes 1]\\
            &= t_0^2 (\partial_a\partial_b - (-1)^{|a||b|}\partial_b\partial_a) \otimes 1 \\
            &= 0
        \end{align*}

        \item[Case 12] $[e_{0b},e_{00}]$\vskip 5pt
        \noindent We have $[e_{0b},e_{00}]=-e_{0b}$, so $\varphi([e_{0b},e_{00}])=-t_0\partial_b \otimes 1$. 
        Since $R$ is central, we also have
        \begin{align*}
            [\varphi(e_{0b}),\varphi(e_{00})] &= [t_0\partial_b\otimes 1, t_0\partial_0 \otimes 1]\\
            &= (-t_0\partial_0t_0\partial_b + t_0\partial_bt_0\partial_0)\otimes 1\\
            &= -( t_0(\partial_0t_0-t_0\partial_0) \partial_b)\otimes 1\\
            &= - t_0\partial_b\otimes 1 = \varphi([e_{0b}, e_{00}]).
        \end{align*}
        \item[Case 13] $[e_{0c}, e_{ab}]$, $c\ne a$\vskip 5pt \noindent
        We have $[e_{0c},e_{ab}]=0$ and
        \begin{align*}
            [\varphi(e_{0c}),\varphi(e_{ab})] &= [t_0\partial_c\otimes 1, t_a\partial_b \otimes 1 + 1\otimes e_{ab} + \delta_{ab}(-1)^{|a||b|}R]\\
            &= [t_0\partial_c, t_a\partial_b] \otimes 1.
        \end{align*}
        Now we have
        \begin{align*}
            [t_0\partial_c, t_a\partial_b]&=t_0\partial_c t_a\partial_b - (-1)^{|c|(|a|+|b|)}t_a\partial_b t_0\partial_c\\
            &= t_0\partial_c t_a\partial_b - (-1)^{|c|(|a|+|b|)+ |c|(|a|+|b|)}t_0\partial_c t_a\partial_b\\
            &= 0.
        \end{align*}
        Thus $[\varphi(e_{0c}), \varphi(e_{ab})] = 0 = \varphi([e_{0c}, e_{ab}])$.
        \item[Case 14] $[e_{ab}, e_{c0}]$, $b\ne c$\vskip 5pt \noindent
        We have $[e_{ab}, e_{c0}]=0$. Also,
        \begin{align*}
            [\varphi(e_{ab}), \varphi(e_{c0})] &= \left[t_a\partial_b\otimes 1 + 1\otimes e_{ab} + \delta_{ab}(-1)^{|a||b|}R, t_c\partial_0 \otimes 1- \sum_{j}(-1)^{|c||j|} \frac{t_j}{t_0}\otimes e_{cj}\right]\\
            &=[t_a\partial_b, t_c\partial_0]\otimes 1-\sum_{j}(-1)^{|c||j|}\left[t_a\partial_b\otimes 1, \frac{t_j}{t_0}\otimes e_{cj}\right]\\&\hspace{20mm} - \sum_{j}(-1)^{|c||j|}\left[1\otimes e_{ab}, \frac{t_j}{t_0}\otimes e_{cj}\right].
        \end{align*}
        Now we evaluate the brackets. We have
        \begin{align*}
            [t_a\partial_b, t_c\partial_0] &= t_a\partial_bt_c\partial_0 - (-1)^{(|a|+|b|)|c|} t_c\partial_0 t_a\partial_b\\
            &= (-1)^{(|a|+|b|)|c|} t_c\partial_0 t_a\partial_b - (-1)^{(|a|+|b|)|c|} t_c\partial_0 t_a\partial_b\\
            &= 0.
        \end{align*}
        Next, we have
        \begin{align*}
            \left[t_a\partial_b\otimes 1, \frac{t_j}{t_0}\otimes e_{cj}\right] &= t_a\partial_b \frac{t_j}{t_0}\otimes e_{cj} - (-1)^{(|a|+|b|)|c|} \left(\frac{t_j}{t_0}\otimes e_{cj}\right)(t_a\partial_b\otimes 1)\\
            &= \frac{1}{t_0}t_a\partial_bt_j\otimes e_{cj} - (-1)^{(|a|+|b|)|j|}\frac{1}{t_0} t_j t_a\partial_b \otimes 1.
        \end{align*}
        If $j\ne b$, then this becomes
        \[\frac{1}{t_0}t_a\partial_bt_j\otimes e_{cj} - \frac{1}{t_0}t_a\partial_bt_j\otimes e_{cj} = 0.\]
        If $j = b$ this becomes
        \[\frac{t_a}{t_0} (\partial_bt_b - (-1)^{|b|}t_b\partial_b)\otimes e_{cb} = \frac{t_a}{t_0}\otimes e_{cb}.\]
        Finally, we have
        \begin{align*}
            \left[1\otimes e_{ab}, \frac{t_j}{t_0}\otimes e_{cj}\right] &= (-1)^{(|a|+|b|)|j|}\frac{t_j}{t_0}\otimes (e_{ab}e_{cj}) - (-1)^{(|a|+|b|)|c|}\frac{t_j}{t_0}\otimes e_{cj}e_{ab}\\
            &=  - (-1)^{(|a|+|b|)|c|}\frac{t_j}{t_0}\otimes e_{cj}e_{ab}
        \end{align*}
        Now $e_{cj}e_{ab}$ is nonzero if and only if $j=a$. When $j=a$, the above expression becomes
        \begin{align*}
            (-1)^{(|a|+|b|)|c|}\frac{t_a}{t_0}\otimes e_{cb}.
        \end{align*}
        Plugging these back into the orginal expression, we get
        \begin{align*}
            [\varphi(e_{ab}), \varphi(e_{c0})]  = - (-1)^{|c||b|}\frac{t_a}{t_0}\otimes e_{cj} + (-1)^{|c||a|+ (|a|+|b|)|c|}\frac{t_a}{t_0}\otimes e_{cb} = 0 = \varphi([e_{ab}, e_{c0}]).
        \end{align*}

    \end{description}
    This completes the verification and finishes the proof. 
\end{proof}

\section{Formulas for the images of certain elements under \texorpdfstring{$\varphi$}{varphi}}\label{long computations}
For the following section, we will write these Gelfand invariants in terms of the following special elements of $U(\gl(m|n))$. Set $r_0^{\gl(m|n)}(a,b)=\delta_{ab}(-1)^{|a||b|}$. For $k\ge 0$ and $a,b\in I$, let
\begin{align}\label{r_k}
r_{k+1}^{\gl(m|n)}(a,b) = \sum_{i_1,\ldots, i_k\in I} (-1)^{|i_1|+\cdots + |i_k|}e_{ai_1}\otimes e_{i_1i_{2}}\otimes \cdots \otimes e_{i_kb}.
\end{align}
Then we have
\[G_k^{\gl(m|n)} = \sum_{i\in I } r_k^{\gl(m|n)}(i,i)\]
for $k\ge 1$. We compute the images of the elements $r_{k}^{\gl(m|n)}(a, b)$ for all $k, a, b$. The proofs of these formulas are not dependent on the results in previous sections of this paper. In this way, we can obtain an alternative (computational) proof of Theorem~\ref{thm : gelfand}.

First, notice the following identity:
\[r_{k+1}^{\gl(m+1|n)}(a,b) = \sum_{i\in \hat{I}} (-1)^{|i|}r_{k}^{\gl(m+1|n)}(a,i)e_{ib}.\]
Indeed, we have
\begin{align*}
    \sum_{i\in \hat{I}} (-1)^{|i|}r_{k}^{\gl(m+1|n)}(a,i)e_{ib} &= \sum_{i\in \hat{I}}\sum_{i_{1},\ldots, i_{k-1}\in \hat{I}} (-1)^{|i_1|+\cdots+|i_{k-1}|+|i|} e_{ai_1}\cdots e_{i_{k-1}i} e_{ib}\\
    &= r_{k+1}^{\gl(m+1|n)}(a,b).
\end{align*}

Define 
\[f_{s}(a,b)=\sum_{i\in I}(-1)^{|i|}t_a\partial_i\otimes r_{s-1}^{\gl(m|n)}(i,b).\]
Also note that we have $f_{1}(a,b) = t_a\partial_b$.

To simplify the statement of the following theorem, we first introduce some more notation. Let 

\begin{align*}
    A_k(s) &= \sum_{g=0}^{k-s}\binom{k}{g} R^g R_2^{k-s-g}\\
    B(s) &= \sum_{i,j\in I} (-1)^{|i|(1+|j|)}\partial_i t_j \otimes r_{s-1}^{\gl(m|n)}(i,j)\\
    D_k(a,b) &= \sum_{g=0}^{k}\binom{k}{g}R^g \left(1\otimes r_{k-g}^{\gl(m|n)}(a,b)\right)\\
    E_k(a) &= \sum_{g=0}^{k-1}\left(\binom{k}{g}R^g\sum_{j>0} (-1)^{|a||j|}\frac{t_j}{t_0}\otimes r_{k-g}^{\gl(m|n)}(a,j)\right)\\
    F_k(a,b) &= \sum_{s=1}^k f_s(a,b)A_k(s).
\end{align*}
Then, we can find the images of the $r_k^{\gl(m+1|n)}$.
\begin{thm}\label{thm: images of r_k}
    We have
    \begin{align*}
        \varphi(r_k^{\gl(m+1|n)}(a,b)) &= F_k(a,b) + D_k(a,b)\\
        \varphi(r_k^{\gl(m+1|n)}(a,0)) &= A_k(1)(t_a\partial_0\otimes 1) - E_k(a) - \left(\frac{t_a}{t_0}\otimes 1\right)\sum_{s=2}^k A_k(s)B(s)\\
        \varphi(r_k^{\gl(m+1|n)}(0,b)) &= F_k(0,b)\\
        \varphi(r_k^{\gl(m+1|n)}(0,0)) &= A_k(1)(t_0\partial_0\otimes 1) + R^k -\sum_{s=2}^k A_k(s)B(s).
    \end{align*}
\end{thm}

\begin{proof}
    We prove all four statements simultaneously by induction on $k$. The base case $k = 1$ follows from the definition of $\varphi$. Suppose the formulas in the statement of the Theorem are true for some positive integer $k$. Let us prove them for $k + 1$.

    First, consider the value of $\varphi(r_k^{\gl(m+1|n)}(a, b))$:
    \begin{align*}
        \varphi(r_{k+1}^{\gl(m|n)}(a,b)) &= \varphi(r_k^{\gl(m+1|n)}(a,0))\varphi(e_{0b}) +\sum_{i\in I}(-1)^{|i|}\varphi(r_{k}^{\gl(m|n)}(a,i))\varphi(e_{ib})\\
        &= \left(A_k(1)(t_a\partial_0\otimes 1) - E_k(a) - \left(\frac{t_a}{t_0}\otimes 1\right)\sum_{s=2}^k A_k(s)B(s)\right)(t_0\partial_b\otimes 1)\\
        &\phantom{=}+\sum_{i\in I}(-1)^{|i|}(F_k(a,i)+D_k(a,i))(t_i\partial_b\otimes 1 + 1\otimes e_{ib} + \delta_{ib}(-1)^{|i||b|}R)\\
        &= A_k(1)(t_a\partial_0 t_0\partial_b\otimes 1) \\
        &\phantom{=}- \sum_{g=0}^{k-1}\left(\binom{k}{g}R^g\sum_{i\in I}(-1)^{(|a|+|b|)|i| + |a||b|}t_i\partial_b\otimes r_{k-g}^{\gl(m|n)}(a,i)\right)\\
        &\phantom{=}-\left(\frac{t_a}{t_0}\otimes 1\right)\sum_{s=2}^kA_k(s)B(s)(t_0\partial_b\otimes 1)+ \sum_{i\in I}(-1)^{|i|}\sum_{s=1}^k f_s(a,i)(t_i\partial_b\otimes 1)A_k(s)\\
        &\phantom{=}+ \sum_{i\in I}\sum_{g=0}^k \binom{k}{g}R^g (-1)^{(|a|+|b|)|i| + |a||b|}\left(t_i\partial_b\otimes r_{k-g}^{\gl(m|n)}(a,i)\right)\\
        &\phantom{=}+ \sum_{i\in I}\sum_{s=1}^k\left((-1)^{|i|}f_s(a,i)(1\otimes e_{ib})A_k(s)\right) \\ 
        &\phantom{=}+\sum_{i\in I}\sum_{g=0}^k\binom{k}{g}R^g (1\otimes (-1)^{|i|}r_{k-g}^{\gl(m|n)}(a,i)e_{ib})\\
        &\phantom{=}+\sum_{s=1}^k f_s(a,b)A_k(s)R + \sum_{g=0}^k\binom{k}{g}R^{g+1}\left(1\otimes r_{k-g}^{\gl(m|n)}(a,b)\right).
    \end{align*}
    Write this sum as $X_1 - X_2 - X_3 + X_4 + X_5 + X_6 + X_7 + X_8 + X_9$. Then 
    \begin{align*}
        X_5 - X_2 = R^k(t_a\partial_b\otimes 1).
    \end{align*}
    We have
    \begin{align*}
        X_3 &= \left(\frac{t_a}{t_0}\otimes 1\right)\sum_{s=2}^kA_k(s)B(s)(t_0\partial_b\otimes 1)\\
        &= \sum_{s=2}^k \sum_{i,j\in I}(-1)^{|i|(1+|j|)}\frac{t_a}{t_0}\partial_i t_j \otimes r_{s-1}^{\gl(m|n)}(i,j)\otimes (t_0\partial_b\otimes 1)A_k(s)\\
        &=\sum_{s=2}^k \sum_{i,j\in I} (-1)^{|i|+|i||j| + |i||b| + |j||b|}t_a\partial_i t_j \partial_b\otimes r_{s-1}^{\gl(m|n)}(i,j)A_k(s)
    \end{align*}
    and
    \begin{align*}
        X_4 &= \sum_{i\in I}(-1)^{|i|}\sum_{s=1}^k f_s(a,i)(t_i\partial_b\otimes 1)A_k(s)\\
        &=\sum_{s=1}^k \sum_{i,j\in I}(-1)^{|i|+|j|}t_a\partial_j \otimes r_{s-1}^{\gl(m|n)}(j,i)\otimes (t_i\partial_b\otimes 1)A_k(s)\\
        &=\sum_{s=1}^k \sum_{i,j\in I}(-1)^{|j| + |j||i| + |j||b| + |i||b|}t_a\partial_j t_i \partial_b\otimes r_{s-1}^{\gl(m|n)}(j,i)A_k(s)\\
        &= \sum_{s=1}^k \sum_{i,j\in I} (-1)^{|i|+|i||j| + |i||b| + |j||b|}t_a\partial_i t_j \partial_b\otimes r_{s-1}^{\gl(m|n)}(i,j)A_k(s).
    \end{align*}
    Thus
    \[X_4 - X_3 = \left(\sum_{i\in I}(-1)^{|i|}t_a\partial_it_i \partial_b\otimes 1\right)A_k(1) = ((t_a\partial_b \otimes 1)(R_2)-t_a\partial_0t_0\partial_b \otimes 1)A_k(1)\]
    as $\sum_{i\in \hat{I}}(-1)^{|i|}t_a\partial_it_i\partial_b\otimes 1 = (t_a\partial_b\otimes 1)R_2$:
    \begin{align*}
        \sum_{i\in \hat{I}} (-1)^{|i|}t_a\partial_i t_i \partial_b &= \sum_{i\in \hat{I}} t_a ((-1)^{|i|}+t_i\partial_i)\partial_b\\
        &= (t_a\partial_b)(m+1-n) +\sum_{\substack{i\in \hat{I}\\ i\ne b}}t_at_i\partial_i\partial_b + t_at_b\partial_b\partial_b\\
        &= (t_a\partial_b)(\mcE+ m+1-n) - (t_a\partial_b)(t_b\partial_b)+ t_at_b\partial_b\partial_b.
    \end{align*}
    If $b$ is even, then the last two terms become $-t_a\partial_b$. If $b$ is odd, the last term is zero, and it is easily verified that $t_a\partial_bt_b\partial_b = t_a\partial_b$, Either way, we get $-(t_a\partial_b)(t_b\partial_b)+t_at_b\partial_b\partial_b = -t_a\partial_b$. Plugging this back in proves the desired identity.
    Then
    \[X_1 + X_4 - X_3 = (t_a\partial_b\otimes 1)A_k(1)R_2 = (t_a\partial_b\otimes 1)\left(\sum_{g=0}^{k-1}\binom{k}{g}R^gR_2^{k-g}\right).\]
    Therefore we have
    \[X_1 + X_4 - X_3 + X_5 - X_2 = (t_a\partial_b\otimes 1)\left(\sum_{g=0}^{k}\binom{k}{g}R^g R_2^{k-g}\right).\]
    Since $\sum_{i\in I}(-1)^{|i|}f_s(a,i)e_{ib} = f_{s+1}(a,b)$, we have
    \[X_6 = \sum_{s=1}^k f_{s+1}(a,b)A_k(s) = \sum_{s= 2}^{k+1}f_s(a,b)\left(\sum_{g=0}^{k+1-s}\binom{k}{g}R^gR_2^{k+1-s-g}\right).\]
    Thus
    \[X_6 + X_1 + X_4 - X_3 + X_5 - X_2 = \sum_{s=1}^{k+1}f_s(a,b)\left(\sum_{g=0}^{k+1-s}\binom{k}{g}R^g R_2^{k+1-s-g}\right).\]
    But
    \[X_8 = \sum_{s=1}^k\left(f_s(a,b)\sum_{g=1}^{k-s+1}\binom{k}{g-1}R^gR_2^{k-s-g+1}\right).\]
    Thus
    \[X_8 +X_6 + X_1 + X_4 - X_3 + X_5 - X_2 = \sum_{s=1}^{k+1}\left(f_s(a,b)\sum_{g=0}^{k+1-s}\binom{k+1}{g}R^g R_2^{k+1-s-g}\right) = F_{k+1}(a,b).\]
    Now, since $\sum_{i\in I}(-1)^{|i|}r_{k-g}^{\gl(m|n)}(a,i)e_{ib}= r_{k+1-g}^{\gl(m|n)}(a,b)$, we find that
    \[ X_7 = \sum_{g=0}^k\binom{k}{g}R^g(1\otimes r_{k+1-g}^{\gl(m|n)}(a,b)).\]
    Also
    \[ X_9 = \sum_{g=1}^k\binom{k}{g-1}R^g(1\otimes r_{k+1-g}^{\gl(m|n)}(a,b))\]
    so
    \[ X_7+ X_9 = \sum_{g=0}^k\binom{k+1}{g}R^g(1\otimes r_{k+1-g}^{\gl(m|n)}(a,b)) = D_{k+1}(a, b).\]

    Therefore
    \[X_1 - X_2 - X_3 + X_4 + X_5 + X_6 + X_8 + (X_7 + X_9) = F_{k+1}(a, b) + D_{k+1}(a, b).\]
    This completes the proof of the inductive step for $\varphi(r_{k+1}^{\gl(m+1|n)}(a,b))$.

    Next, we consider the value of $\varphi(r_{k}^{\gl(m+1|n)}(a,0))$. We have
    \begin{align*}
        \varphi(r_{k+1}^{\gl(m+1|n)}(a,0)) &= \varphi(r_{k}^{\gl(m+1|n)}(a,0))\varphi(e_{00}) + \sum_{i\in I}(-1)^{|i|}\varphi(r_{k}^{\gl(m+1|n)}(a,i))\varphi(e_{i0})\\
        &= \left(A_k(1)(t_a\partial_0\otimes 1)-E_k(a)- \left(\frac{t_a}{t_0}\otimes 1\right)\sum_{s=2}^{k}A_k(s)B(s)\right)(t_0\partial_0\otimes 1 + R)\\
        &\phantom{=}+\sum_{i\in I}(-1)^{|i|}(F_k(a,i)+D_k(a,i))\left(t_i\partial_0\otimes 1 - \sum_{j\in I}(-1)^{|i||j|}\frac{t_j}{t_0}\otimes e_{ij}\right)\\
        &= A_k(1)(t_a\partial_0t_0\partial_0\otimes 1)-\sum_{g=0}^{k-1}\left(\binom{k}{g}\sum_{j\in I}(-1)^{|a||j|}t_j\partial_0\otimes r_{k-g}^{\gl(m|n)}(a,j)R^g\right)\\
        &\phantom{=} - \sum_{s=2}^k A_k(s) \left(\sum_{i,j\in I}(-1)^{|i|(1+|j|)}t_a\partial_i t_j \partial_0\otimes r_{s-1}^{\gl(m|n)}(i,j)\right)\\
        &\phantom{=}+A_k(1)R(t_a\partial_0\otimes 1) - \sum_{g=0}^{k-1}\left(\binom{k}{g}R^{g+1}\sum_{j\in I} (-1)^{|a||j|}\frac{t_j}{t_0}\otimes r_{k-g}^{\gl(m|n)}(a,j)\right)\\
        &\phantom{=}-\left(\frac{t_a}{t_0}\otimes 1\right)\sum_{s=2}^kA_k(s)RB(s)\\
        &\phantom{=}+\sum_{s=1}^k\left(\sum_{i\in I}(-1)^{|i|}f_s(a,i)(t_i\partial_0\otimes 1)\right)A_k(s)\\
        &\phantom{=}+ \sum_{g=0}^k\binom{k}{g}R^g \left(\sum_{i\in I}(-1)^{|a||i|}t_i\partial_0\otimes r_{k-g}^{\gl(m|n)}(a,i)\right)\\
        &\phantom{=}-\sum_{s=1}^k\left(\sum_{i,j\in I}(-1)^{|i|(1+|j|)}f_s(a,i) \left(\frac{t_j}{t_0}\otimes e_{ij}\right)\right)A_k(s)\\
        &\phantom{=}-\sum_{g=0}^k \binom{k}{g}R^g\left(\sum_{i,j\in I}(-1)^{|a||j|}\frac{t_j}{t_0}\otimes (-1)^{|i|}r^{\gl(m|n)}_{k-g}(a,i)e_{ij}\right).
    \end{align*}
    Call this sum $X_1 - X_2 - X_3 + X_4 - X_5 - X_6 + X_7 + X_8 - X_9 - X_{10}$. We have 
    \[X_8- X_2 = R^k (-1)^{|a|}t_a\partial_0 \otimes r_0^{\gl(m|n)}(a,a) = R^k(t_a\partial_0 \otimes 1).\]
    Also
    \begin{align*}
        X_5 + X_{10} &= \sum_{g=1}^k \binom{k}{g-1}R^g\left(\sum_{j\in I} (-1)^{|a||j|} \frac{t_j}{t_0}\otimes r_{k+1-g}^{\gl(m|n)}(a,j)\right)\\
        &\phantom{=}+\sum_{g=0}^k \binom{k}{g}R^g\left(\sum_{j\in I} (-1)^{|a||j|} \frac{t_j}{t_0}\otimes r_{k+1-g}^{\gl(m|n)}(a,j)\right)\\
        &= \sum_{g=0}^k \binom{k+1}{g}R^g\left(\sum_{j\in I} (-1)^{|a||j|} \frac{t_j}{t_0}\otimes r_{k+1-g}^{\gl(m|n)}(a,j)\right)\\
        &= E_{k+1}(a).
    \end{align*}
    Next, 
    \[X_9 = \sum_{s=1}^k\left(\sum_{i,j\in I}(-1)^{|i|(1+|j|)}f_s(a,i) \left(\frac{t_j}{t_0}\otimes e_{ij}\right)\right)A_k(s).\]
    For any positive integer $s$, 
    \begin{align*}
        &\sum_{i,j\in I}(-1)^{|i|(1+|j|)} f_s(a,i)\left(\frac{t_j}{t_0}\otimes e_{ij}\right)\\ =& \sum_{i,j\in I}(-1)^{|i|(1+|j|)}\left(\sum_{\ell\in I} (-1)^{|\ell|} t_a\partial_{\ell} \otimes r_{s-1}^{\gl(m|n)}(\ell, i)\right)\left(\frac{t_j}{t_0}\otimes e_{ij}\right)\\
        =& \sum_{j,\ell\in I} (-1)^{|\ell|(1+|j|)} t_a\partial_{\ell}t_jt_0^{-1}\otimes\left(\sum_{i\in I}(-1)^{|i|}r_{s-1}^{\gl(m|n)}(\ell, i) e_{ij}\right)\\
        =& \left(\frac{t_a}{t_0}\otimes 1\right)\sum_{j,\ell\in I}(-1)^{|\ell|(1+|j|)}\partial_\ell t_j\otimes r_s^{\gl(m, n)}(\ell, j)\\
        =&  \left(\frac{t_a}{t_0}\otimes 1\right)B(s+1).
        \end{align*}
    Then
    \[X_9 =  \left(\frac{t_a}{t_0}\otimes 1\right) \sum_{s=2}^{k+1}B(s)\left(\sum_{g=0}^{k+1-s}\binom{k}{g}R^g R_2^{k+1-s-g}\right).\]
    Also
    \[X_6 =  \left(\frac{t_a}{t_0}\otimes 1\right) \sum_{s=2}^{k}B(s)\left(\sum_{g=1}^{k+1-s}\binom{k}{g-1}R^g R_2^{k+1-s-g}\right).\]
    Adding, we get
    \[X_6 + X_9 =   \left(\frac{t_a}{t_0}\otimes 1\right) \sum_{s=2}^{k}A_{k+1}(s)B(s).\]

    Next, for any positive integer $s$, 
    \begin{align*}
        \sum_{i\in I} (-1)^{|i|}f_s(a,i)(t_i\partial_0\otimes 1) &= \sum_{i,j\in I}(-1)^{|i||j|}t_a\partial_j t_i\partial_0\otimes r_{s-1}^{\gl(m|n)}(j,i)\\
        &=\sum_{i,j\in I}(-1)^{|i||j|}t_a\partial_it_j\partial_0\otimes r_{s-1}^{\gl(m|n)}(i,j).
    \end{align*}
    Thus
    \[X_7 = \sum_{s=1}^k\left(\sum_{i,j\in I}(-1)^{|i||j|}t_a\partial_it_j\partial_0\otimes r_{s-1}^{\gl(m|n)}(i,j)\right)A_k(s).\]
    On the other hand,
    \[X_3 = \sum_{s=2}^k\left(\sum_{i,j\in I}(-1)^{|i||j|}t_a\partial_it_j\partial_0\otimes r_{s-1}^{\gl(m|n)}(i,j)\right)A_k(s)\]
    so
    \[X_7 - X_3 = A_k(1)(t_a\partial_0\otimes 1)(R_2-t_0\partial_0\otimes 1).\]
    Then
    \[X_7 - X_3 + X_1 = A_k(1)R_2(t_a\partial_0\otimes 1) = \left(\sum_{g=0}^{k-1} \binom{k}{g}R^g R_2^{k-g}\right)(t_a\partial_0\otimes 1).\]
    Thus
    \[X_7 - X_3 + X_1 + X_8 - X_2 = \left(\sum_{g=0}^{k} \binom{k}{g}R^g R_2^{k-g}\right)(t_a\partial_0\otimes 1).\]
    Since
    \[X_4 = \left(\sum_{g=1}^k \binom{k}{g-1}R^g R_2^{k-g}\right)(t_a\partial_0\otimes 1),\]
    we get
    \[X_7 - X_3 + X_1 + X_8 - X_2 + X_4 = \left(\sum_{g=0}^k\binom{k+1}{g}R^g R_2^{k-g}\right)(t_a\partial_0\otimes 1) = A_{k+1}(1)(t_a\partial_0\otimes 1).\]
    Thus
    \begin{align*}
        \varphi(r_{k+1}^{\gl(m+1|n)}(a,0)) &= (X_7 - X_3 + X_1 + X_8 - X_2 + X_4) - (X_5 + X_{10}) - (X_6+X_9)\\
        &= A_{k+1}(1)(t_a\partial_0\otimes 1) - E_{k+1}(a) - \left(\frac{t_a}{t_0}\otimes 1\right)\sum_{s=2}^k A_{k+1}(s)B(s).
    \end{align*}
    This completes the inductive step for $r_{k+1}^{\gl(m|n)}(a,0)$.

    Next, we consider the value of $ \varphi(r_{k}^{\gl(m+1|n)}(0, b))$. We have
    \begin{align*}
        \varphi(r_{k+1}^{\gl(m+1|n)}(0,b)) 
        &= \varphi(r_k^{\gl(m+1|n)}(0, 0))\varphi(e_{0b}) +\sum_{i\in I}(-1)^{|i|}\varphi(r_k^{\gl(m+1|n)}(0,i))\varphi(e_{ib})\\
        &= \left(A_k(1)(t_0\partial_0\otimes 1)+R^k - \sum_{s=2}^k A_k(s)B(s)\right)(t_0\partial_b\otimes 1) \\
        &\phantom{=} + \sum_{i\in I}(-1)^{|i|} F_k(0,i)(t_i\partial_b \otimes 1 + 1\otimes e_{ib}+\delta_{ib}(-1)^{|i||b|}R)\\
        &=A_k(1)(t_0\partial_0t_0\partial_b\otimes 1) + R^k(t_0\partial_b\otimes 1) - \sum_{s=2}^{k}A_k(s)B(s)(t_0\partial_b\otimes 1) \\
        &\phantom{=}+\sum_{s=1}^k\sum_{i\in I} (-1)^{|i|}f_s(0,i)(t_i\partial_b\otimes 1)A_k(s)\\ &\phantom{=} + \sum_{s=1}^k\sum_{i\in I}(-1)^{|i|}f_s(0,i)(1\otimes e_{ib})A_k(s)
        +\sum_{s=1}^k f_s(0,b) A_k(s)R.
    \end{align*}
    Write this sum as $X_1 +X_2-X_3+X_4+X_5+X_6$. Now 
    \begin{align*}
        X_3 &= \sum_{s=2}^k \sum_{i,j\in I} (-1)^{|i|(1+|j|)}(\partial_i t_j \otimes r_{s-1}^{\gl(m|n)}(i,j))(t_0\partial_b\otimes 1)A_k(s)\\
        &= \sum_{s=2}^k \sum_{i,j\in I} (-1)^{|i|(1+|j|) + (|i|+|j|)|b|}(t_0\partial_i t_j \partial_b \otimes r_{s-1}^{\gl(m|n)}(i,j))A_k(s)\\
        &= \sum_{s=2}^k \sum_{i,j\in I} (-1)^{(|i|+|j|)(|i|+|b|)}(t_0\partial_i t_j \partial_b \otimes r_{s-1}^{\gl(m|n)}(i,j))A_k(s).
    \end{align*}
    Also 
     \begin{align*}
        X_4 &= \sum_{s=1}^k\sum_{i,j\in I}  (-1)^{|i|+|j|}(t_0\partial_j\otimes r^{\gl(m|n)}(j,i) \otimes t_i\partial_b \otimes 1)A_k(s)\\
        &=\sum_{s=1}^k\sum_{i,j\in I} (-1)^{|i|+|j|+(|i|+|j|)(|i|+|b|)} (t_0\partial_jt_i\partial_b\otimes r_{s-1}^{\gl(m|n)}(j,i) )A_k(s).
    \end{align*}
    Now $(-1)^{|i|+|j|+(|i|+|j|)(|i|+|b|)}= (-1)^{|j| + |j||i|+|i||b|+|j||b|}= (-1)^{(|j|+|i|)(|j|+|b|)}$. Then swapping the variable names $i,j$ gives
    \[X_4 = \sum_{s=1}^k\sum_{i,j\in I}(-1)^{(|i|+|j|)(|i|+|b|)}(t_0\partial_i t_j\partial_b \otimes r_{s-1}^{\gl(m|n)}(i,j))A_k(s).\]
    Then $X_4-X_3$ simplifies to 
    \[\left(\sum_{i\in I} (-1)^{|i|}t_0\partial_it_i\partial_b\otimes 1\right)A_k(1).\]
    Then $X_1+X_4-X_3$ is
    \[(t_0\partial_b\otimes 1)\sum_{g=0}^{k-1}\binom{k}{g}R^g R_2^{k-g}.\]
    Thus
    \[X_2 + X_1 + X_4 - X_3  = (t_0\partial_b\otimes 1)\left(\sum_{g=0}^k \binom{k}{g}R^gR_2^{k-g}\right).\]
    Since $\sum_{i\in I}(-1)^{|i|}f_s(0,i)(1\otimes e_{ib}) = f_{m+1}(0,i)$, we have 
    \[X_5 = \sum_{s=1}^k f_{s+1}(0,b)A_k(s) = \sum_{s=2}^{k+1} f_s(0,b)\sum_{g=0}^{k+1-s}\binom{k}{g}R^gR_2^{k+1-s-g}.\]
    Thus
    \[X_5 + X_2 + X_1 + X_4 - X_3 = \sum_{s=1}^{k+1} f_s(0,b) \sum_{g=0}^{k+1-s}\binom{k}{g}R^g R_2^{k+1-s-g}.\]
    Also
    \[X_6 = \sum_{s=1}^{k}f_s(0,b)\sum_{g=1}^{k+1-s}\binom{k}{g-1}R^gR_2^{k+1-s-g}.\]
    Adding the previous two equations gives 
    \[X_1 + X_2 - X_3 + X_4 + X_5+ X_6 = \sum_{s=1}^{k+1}f_s(0,b)A_{k+1}(s)\]
    which completes the inductive step for the value of $\varphi(r_{k+1}^{\gl(m+1|n)}(0,b))$.

    Finally, we consider the value of $\varphi(r_k^{\gl(m+1|n)}(0,0))$. We have
    \begin{align*}
        \varphi(r_{k+1}^{\gl(m+1|n)}(0,0)) &= \varphi(r_k^{\gl(m+1|n)}(0,0))\varphi(e_{00}) + \sum_{i\in I}(-1)^{|i|}\varphi(r_k^{\gl(m+1|n)}(0,i)) \varphi(e_{i,0})\\
        &=\left(A_k(1)(t_0\partial_0\otimes 1) + R^k - \sum_{s=2}^k A_k(s)B(s)\right)(t_0\partial_0 \otimes 1+R)\\
        &\phantom{=} + \sum_{i\in I}(-1)^{|i|}F_k(0,i)\left(t_i\partial_0 -\sum_{j\in I} (-1)^{|i||j|} \frac{t_j}{t_0}\otimes e_{ij}\right)\\
        &=A_k(1)(t_0\partial_0t_0\partial_0\otimes 1) + R^{k}(t_0\partial_0\otimes 1) -\sum_{s=2}^k A_k(s)B(s)(t_0\partial_0\otimes 1)\\
        &\phantom{=} + A_k(1)R (t_0\partial_0\otimes 1) + R^{k+1}-\sum_{s=2}^{k}\left(A_k(s)RB(s)\right)\\
        &\phantom{=}+\sum_{i\in I} (-1)^{|i|}\sum_{s=1}^kf_s(0,i)(t_i\partial_0\otimes 1)A_k(s)\\
        &\phantom{=}-\sum_{i\in I}(-1)^{|i|}\sum_{s=1}^k f_s(0,i)A_k(s)\sum_{j\in I}(-1)^{|i||j|}\frac{t_j}{t_0}\otimes e_{ij}.
    \end{align*}
    Write this as $X_1 + X_2 - X_3 + X_4 + X_5 - X_6 + X_7 - X_8$. 

    Since 
    \begin{align*}
        X_7 &= \sum_{s=1}^k\sum_{i,j\in I} (-1)^{|i|+|j|}(t_0\partial_j\otimes r_{s-1}^{\gl(m|n)}(j,i)\otimes t_i\partial_0\otimes 1)A_k(s)\\
        &=\sum_{s=1}^k\sum_{i,j\in I}(-1^{|j|(1+|i|)})\partial_j t_i\otimes r_{s-1}^{\gl(m|n)}(j,i) A_k(s)(t_0\partial_0\otimes 1)\\
        &=\sum_{s=1}^k\sum_{i,j\in I}(-1^{|i|(1+|j|)})\partial_i t_j\otimes r_{s-1}^{\gl(m|n)}(i,j) A_k(s)(t_0\partial_0\otimes 1)
    \end{align*}
    we have 
    \begin{align*}
        X_7-X_3 = \sum_{i\in I}(-1)^{|i|}(\partial_i t_i\otimes 1)(t_0\partial_0\otimes 1) A_k(1) = (R_2-(t_0\partial_0\otimes 1))(t_0\partial_0\otimes 1)A_k(1).
    \end{align*}
    Then, expanding out the $A_k(s)$ terms gives
    \begin{align*}
        X_7-X_3 + X_1 + X_2 = \left(\sum_{g=0}^k \binom{k}{g}R^gR_2^{k-g}\right)(t_0\partial_0\otimes 1).
    \end{align*}
    But since $X_4 = \left(\sum_{g=1}^k \binom{k}{g-1}R^g R_2^{k-g}\right)(t_0\partial_0\otimes 1)$, 
    \begin{align*}
        X_7 - X_3 + X_1 + X_2 + X_4 = \left( \sum_{g=0}^k \binom{k+1}{g}R^g R_2^{k-g}\right)(t_0\partial_0\otimes 1) = A_{k+1}(1)(t_0\partial_0\otimes 1).
    \end{align*}

    We have 
    \begin{align*}
        X_8 = \sum_{s=1}^k\left(\sum_{i,j\in I}(-1)^{|i|(1+|j|)}f_s(0,i)\frac{t_j}{t_0}\otimes e_{ij}\right)A_k(s).
    \end{align*}

    For any positive integer $s$, 
    \begin{align*}
        \sum_{i,j\in I}(-1)^{|i|(1+|j|)}f_s(0,i)\frac{t_j}{t_0}\otimes e_{ij} &= \sum_{i,j\in I}(-1)^{|i|(1+|j|)}\sum_{\ell\in I} (-1)^{|\ell|}t_0\partial_{\ell}\otimes r_{s-1}^{\gl(m|n)}(\ell, i) \otimes \frac{t_j}{t_0}\otimes e_{ij}\\ 
        &=\sum_{i,j}\sum_{\ell \in I}(-1)^{|\ell|(1+ |j|)}\partial_{\ell}t_j\otimes (-1)^{|i|}( r_{s-1}^{\gl(m|n)}(\ell, i) e_{ij})\\
        &= \sum_{j,\ell\in I}(-1)^{|\ell|(1+|j|)} \partial_{\ell}t_j\otimes r_{s}^{\gl(m|n)}(\ell, j)\\
        &=B(s+1).
    \end{align*}
    Then
    \[X_8 = \sum_{s=2}^{k+1} \left(\sum_{g=0}^{k+1-s}\binom{k}{g}R^g R_2^{k+1-s-g}\right)B(s).\]
    Also,
    \[X_6 = \sum_{s=2}^{k} \left(\sum_{g=1}^{k+1-s}\binom{k}{g-1}R^g R_2^{k+1-s-g}\right)B(s)\]
    so
    \[X_6 + X_8 = \sum_{s=2}^{k+1}A_{k+1}(s) B(s).\]
    Finally, $X_5 = R^{k+1}$, so combining the above gives
    \[X_1 + X_2 - X_3 + X_4 + X_5 - X_6 + X_7 - X_8 = A_{k+1}(1)(t_0\partial_0\otimes 1) + R^{k+1} - \sum_{s=2}^{k+1}A_{k+1}(s)B(s).\]
    This completes the inductive step for $\varphi(r_{k+1}^{\gl(m+1|n)}(0,0))$, and hence the proof of the theorem.
\end{proof}

\date{\today}
\end{document}